\documentclass[12pt]{article}

\usepackage{amssymb}
\usepackage{amsthm}

\newtheorem{theorem}{Theorem}
\newtheorem{prop}{Proposition}
\newtheorem{lem}{Lemma}
\newtheorem{corol}{Corollary}

\newtheorem{remark}{Remark}
\newtheorem{example}{Example}
\DeclareSymbolFont{AMSb}{U}{msb}{M}{n}
\DeclareMathSymbol{\Eout}{\mathbin}{AMSb}{"45}
\DeclareMathSymbol{\PR}{\mathbin}{AMSb}{"50}
\DeclareMathSymbol{\N}{\mathbin}{AMSb}{"4E}
\DeclareMathSymbol{\R}{\mathbin}{AMSb}{"52}
\DeclareMathSymbol{\bX}{\mathbin}{AMSb}{"58}
\DeclareMathSymbol{\bY}{\mathbin}{AMSb}{"59}
\DeclareMathSymbol{\bZ}{\mathbin}{AMSb}{"5A}

\newcommand{\ca}{{\cal A}}
\newcommand{\cd}{{\cal D}}
\newcommand{\cf}{{\cal F}}

\newcommand{\cu}{{\cal U}}
\newcommand{\cx}{{\cal X}}

\newcommand{\Ftil}{\tilde{F}}
\newcommand{\Gtil}{G_m}
\newcommand{\Htil}{\tilde{H}}

\newcommand{\ptil}{\tilde{p}}
\newcommand{\Ptil}{\tilde{P}}

\newcommand{\hhbar}{{\bar{h}}}
\newcommand{\Hbar}{{\bar{H}}}
\newcommand{\vbar}{{\bar{v}}}

\newcommand{\Pihat}{{\hat{\Pi}}}

\newcommand{\bfu}{{\bf u}}
\newcommand{\bfv}{{\bf v}}

\newcommand{\bfU}{{\bf U}}
\newcommand{\bfV}{{\bf V}}

\newcommand{\E}{{\hbox{\bf E}}}
\newcommand{\Emu}{{\hbox{\bf E}_\mu}}
\newcommand{\EP}{{\hbox{\bf E}_P}}
\newcommand{\fndot}{{\,\cdot\,}}
\newcommand{\indic}{{\hbox{\bf 1}}}

\newcommand{\supp}{{\rm supp}}
\newcommand{\Jeps}{J_\epsilon}
\newcommand{\half}{{\frac{1}{2}}}

\title{A Class of Non-Parametric Statistical Manifolds modelled on Sobolev Space}

\author{Nigel J.~Newton
  \thanks{School of Computer Science and Electronic Engineering, University of Essex,
    Wivenhoe Park, Colchester, CO4 3SQ, United Kingdom. njn@essex.ac.uk}}

\begin{document}

\maketitle

\begin{abstract}
We construct a family of non-parametric (infinite-dimensional) manifolds of finite
measures on $\R^d$.  The manifolds are modelled on a variety of weighted Sobolev
spaces, including Hilbert-Sobolev spaces and mixed-norm spaces.  Each supports the
Fisher-Rao metric as a weak Riemannian metric.  Densities are expressed in terms of
a {\em deformed exponential} function having linear growth.  Unusually for the Sobolev
context, and as a consequence of its linear growth, this ``lifts'' to a nonlinear
superposition (Nemytskii) operator that acts continuously on a particular class of
mixed-norm model spaces, and on the fixed norm space $W^{2,1}$; i.e.~it maps each of
these spaces continuously into itself.  It also maps continuously between other
fixed-norm spaces with a loss of Lebesgue exponent that increases with the number of
derivatives.  Some of the results make essential use of a
log-Sobolev embedding theorem.  Each manifold contains a smoothly embedded
submanifold of {\em probability} measures.  Applications to the stochastic partial
differential equations of nonlinear filtering (and hence to the Fokker-Planck
equation) are outlined.


\end{abstract}

\section{Introduction} \label{se:intro}

In recent years there has been rapid progress in the theory of {\em information
geometry}, and its application to a variety of fields including {\em asymptotic
statistics}, {\em machine learning}, {\em signal processing} and {\em statistical
mechanics}. (See, for example, \cite{niba2,niba3}.) Beginning with C.R.~Rao's
observation that the Fisher information can be interpreted as a Riemannian metric
\cite{rao1}, information geometry has exploited the formalism of manifold theory in
problems of statistical estimation.  The finite-dimensional (parametric) theory is
now mature, and is treated pedagogically in \cite{amar1,barn1,chen1,laur1,muri1}.
The archetypal example is the finite-dimensional {\em exponential model}, which is
based on a finite set of real-valued random variables defined on an underlying
probability space $(\bX,\cx,\mu)$.  Affine combinations of these are exponentiated
to yield probability density functions with respect to the reference measure $\mu$.
This construction induces a topology on the resulting set of probability measures,
that is compatible with the statistical divergences of estimation theory, derivatives
of which can be used to define the Fisher-Rao metric and covariant derivatives
having various statistical interpretations.

The first successful extension of these ideas to the non-parametric setting appeared
in \cite{pise1}, and was further developed in \cite{gipi1,piro1,cepi1}.  These papers
follow the formalism of the exponential model by using the log of the density as a
chart.  This approach requires a model space with a strong topology: the exponential
Orlicz space.  It has been extended in a number of ways.  In \cite{loqu1}, the
exponential function is replaced by the so-called {\em $q$-deformed exponential},
which has an important interpretation in statistical mechanics.  (See chapter 7 in
\cite{naud1}.)  The model space used there is $L^\infty(\mu)$.  A more general class
of deformed exponential functions is used in \cite{vica1} to construct families of
probability measures dubbed {\em $\varphi$-families}.  The model spaces used are
Musielak-Orlicz spaces.

One of the most important statistical divergences is the Kullback-Leibler
(KL) divergence.  For probability measures $P$ and $Q$ having densities $p$ and $q$
with respect to $\mu$, this is defined as follows:
\begin{equation}
\cd(P|Q) = \int p\log(p/q)d\mu. \label{eq:KLdivgen}
\end{equation}
The KL divergence can be given the bilinear representation
$\langle p,\log p-\log q\rangle$, in which probability densities and their logs take
values in dual function spaces (for example, the Lebesgue spaces $L^\lambda(\mu)$ and
$L^{\lambda/(\lambda-1)}(\mu)$ for some $1<\lambda<\infty$).  Loosely speaking, in
order for the KL divergence to be smooth on an infinite-dimensional manifold, the
charts of the latter must ``control'' both the density $p$ and its log, and this
provides one explanation of the need for strong topologies on the model spaces of
non-parametric exponential models.  This observation led to the construction
in \cite{newt4} of an infinite-dimensional statistical manifold modelled
on Hilbert space.  This employs a ``balanced chart'' (the sum of the density and
its log), which directly controls both.  This chart was later used in \cite{newt6}
in the development of Banach manifolds modelled on the Lebesgue spaces
$L^\lambda(\mu)$, for $\lambda\in[2,\infty)$.  These give increasing degrees of
smoothness to statistical divergences.  An ambient manifold of {\em finite} measures
was also defined in \cite{newt6}, and used in the construction of $\alpha$-parallel
transport on the embedded statistical manifold.

These manifolds make no reference to any topology that the underlying sample space
$\bX$ may possess.  Statistical divergences measure dependency between abstract random
variables (those taking values in measurable spaces) without reference to any other
structures that these spaces may have.  Nevertheless, topologies, metrics and linear
structures on $\bX$ play important roles in many applications.  For example, the
Fokker-Planck and Boltzmann equations both quantify the evolution of probability
density functions on $\R^d$, making direct reference to the latter's topology through
differential operators.  A natural direction for research in infinite-dimensional
information geometry is to adapt the manifolds outlined above to such problems by
incorporating the topology of the sample space in the model space.  One way of
achieving this is to use model spaces of Sobolev type.  This is carried out in the
context of the exponential Orlicz manifold in \cite{lopi1}, where it is applied to the
spatially homogeneous Boltzmann equation.  Manifolds modelled on the Banach spaces
$C_b^k(B;\R)$, where $B$ is an open subset of an underlying (Banach) sample space,
are developed in \cite{newt8}, and manifolds modelled on Fr\'{e}chet spaces of
smooth densities are developed in \cite{babm1,brmi1} and \cite{newt8}.

The aim of this paper is to develop Sobolev variants of the Lebesgue $L^\lambda(\mu)$
manifolds of \cite{newt4,newt6} when the sample space $\bX$ is $\R^d$.  Our
construction includes, as a special case, a class of Hilbert-Sobolev
manifolds. In developing these, the author was motivated by applications in
{\em nonlinear filtering}.  The equations of nonlinear filtering for
diffusion processes generalise the Fokker-Planck equation by adding a term that
accounts for partial observations of the diffusion.  Let $(X_t,Y_t,t\ge 0)$ be a
$d+1$-vector Markov diffusion process defined on a probability space
$(\Omega,\cf,\PR)$, and satisfying the It\^{o} stochastic differential equation
\begin{equation}
d\left[\begin{array}{c}X_t \\ Y_t\end{array}\right]
  = \left[\begin{array}{c} f(X_t) \\ h(X_t)\end{array}\right]dt
    + \left[\begin{array}{cc} g(X_t) & 0 \\ 0 & 1\end{array}\right]dV_t,
      \label{eq:NLFsetup}
\end{equation}
where $Y_0=0$, $(V_t,t\ge 0)$ is a $d+1$-vector standard Brownian motion, independent
of $X_0$, and $f:\R^d\rightarrow\R^d$, $g:\R^d\rightarrow\R^{d\times d}$ and
$h:\R^d\rightarrow\R$ are suitably regular functions.  The nonlinear filter for $X$
computes, at each time $t$, the conditional probability distribution of $X_t$ given
the history of the {\em observations process} $(Y_s,0\le s\le t)$.  Since $X$ and $Y$
are jointly Markov the nonlinear filter can be expressed in a time-recursive manner.
Under suitable technical conditions, the observation-conditional distribution of
$X_t$ admits a density, $p_t$, (with respect to Lebesgue measure) satisfying the
{\em Kushner Stratonovich} stochastic partial differential equation \cite{crro1}
\begin{equation}
dp_t = \ca p_t\, dt
          + p_t(h-\hat{h}_t) d(Y_t-\hat{h}_tdt), \label{eq:kusteq}
\end{equation}
where $\ca$ is the Kolmogorov forward (Fokker-Planck) operator for $X$, and
$\hat{h}_t$ is the $(Y_s,0\le s\le t)$-conditional mean of $h(X_t)$.

The exponential Orlicz manifold was proposed as an ambient manifold for partial
differential equations of this type in \cite{brpi1} (and the earlier references
therein), and methods of projection onto submanifolds were developed.  Applications
of the Hilbert manifold of \cite{newt4} to nonlinear filtering were
developed in \cite{newt5,newt7}, and information-theoretic properties were
investigated.

It was argued in \cite{newt6,newt7} that statistical divergences such as the KL
divergence are natural measures of error for approximations to Bayesian conditional
distributions such as those of nonlinear filtering.  This is particularly so when
the approximation constructed is used to estimate a number statistics of the process
$X$, or when the dynamics of $X$ are significantly nonlinear.  We summarise these
ideas here since they motivate the developments that follow; details can be found
in \cite{newt7}.  If our purpose is to estimate a single real-valued variate
$v(X_t)\in L^2(\mu)$, then the estimate with the minimum mean-square error is the
conditional mean $\vbar_t:=\E_{\Pi_t}v=\Eout(v(X_t)|(Y_s,0\le s\le t))$, where
$\Eout$ is expectation with respect to $\PR$, and $\Pi_t$ is the conditional
distribution of $X_t$.  If the estimate is based on a $(Y_s,0\le s\le t)$-measurable
approximation to $\Pi_t$, $\Pihat_t$, then the mean-square error admits the orthogonal
decomposition
\begin{equation}
\Eout(v(X_t)-\E_{\Pihat_t}v)^2
   =  \Eout\E_{\Pi_t}(v-\vbar_t)^2 + \Eout(\vbar_t-\E_{\Pihat_t}v)^2.
      \label{eq:l2decom}
\end{equation}
The first term on the right-hand side here is the {\em statistical error}, and is
associated with the limitations of the observation $Y$; the second term is the
{\em approximation error} resulting from the use of $\Pihat_t$ instead of $\Pi_t$.
When comparing different approximations, it is appropriate to measure the second term
{\em relative to the first}; if $\vbar_t$ is a poor estimate of $v(X_t)$ then there
is no point in approximating it with great accuracy.  Maximising these {\em relative}
errors over all square-integrable variates leads to the (extreme) multi-objective
measure of mean-square approximation errors $\cd_{MO}(\Pihat_t|\Pi_t)$, where
\begin{equation}
\cd_{MO}(Q|P)
  := \half\sup_{v\in L^2(P)}\frac{(\E_Q v-\EP v)^2}{\EP(v-\EP v)^2}
   = \half\|dQ/dP-1\|_{L^2(P)}^2. \label{eq:chisqd}
\end{equation}
$\cd_{MO}$ is Pearson's $\chi^2$-divergence.  Although extreme, it illustrates an
important feature of multi-objective measures of error---they require probabilities
of events that are small to be approximated with greater absolute accuracy than
those that are large.  A less extreme multi-objective measure of mean-square errors
is developed in \cite{newt7}.  This constrains the functions $v$ of (\ref{eq:chisqd})
to have exponential moments.  The resulting measure of errors is shown to be of class
$C^1$ on the Hilbert manifold of \cite{newt4}, and so has this same property on the
manifolds developed here.  See \cite{newt7} for further discussion of these ideas. 

The paper is structured as follows.  Section \ref{se:modspa} provides the technical
background in mixed-norm weighted Sobolev spaces, where the $L^\lambda$ spaces are
based on a probability measure.  Section \ref{se:manifmt} constructs $(M,G,\phi)$,
a manifold of finite measures modelled on the general Sobolev space of section
\ref{se:modspa}. It outlines the properties of mixture and exponential representations
of measures on the manifold, as well as those of the KL divergence.  In doing so, it
defines the Fisher-Rao metric and Amari-Chentsov tensor.  Section \ref{se:gtil}
shows that a particular choice of mixed-norm Sobolev space is especially suited to
the manifold in the sense that the density of any $P\in M$ also belongs to the
model space, and the associated nonlinear superposition operator is continuous---a
rare property in the Sobolev context \cite{rusi1}.  Section \ref{se:fixnor} shows
that this property does not hold for fixed norm spaces, except in the special case
$G=W^{2,1}$.  It also develops a general class of fixed norm spaces, for which the
continuity property can be retained if the Lebesgue exponent in the range space is
suitably reduced.  Section \ref{se:manifm} develops an embedded submanifold of
{\em probability} measures $(M_0,G_0,\phi_0)$, in which the charts are {\em centred}
versions of $\phi$. Section \ref{se:NLF} outlines applications to the problem of
nonlinear filtering for a diffusion process, as defined in (\ref{eq:NLFsetup}) and
(\ref{eq:kusteq}).  Finally, section \ref{se:conclu} makes some concluding remarks,
discussing, in particular, a variant of the results that uses the Kaniadakis deformed
logarithm as a chart.

\section{The Model Spaces} \label{se:modspa} 

For some $t\in(0,2]$, let $\theta_t:[0,\infty)\rightarrow[0,\infty)$ be a strictly
increasing function that is twice continuously differentiable on $(0,\infty)$, such
that $\lim_{z\downarrow 0}\theta_t^\prime(z)<\infty$, and
\begin{equation}
\theta_t(z)
   = \left\{\begin{array}{ll}
     0         & {\rm if\ }z=0 \\
     c_t + z^t & {\rm if\ }z\ge z_t
     \end{array}\right\}, \quad{\rm where\ }z_t\ge 0,{\rm\ and\ }c_t\in\R.
     \label{eq:thetadef}
\end{equation}
If $t\in(1,2]$ then we also require $\theta_t$ and $-\sqrt{\theta_t}$ to be convex.

\begin{example} \label{ex:muex}
\begin{enumerate}
\item[(i)] Simple: $t\in[1,2]$ and $z_t=c_t=0$.
\item[(ii)] Smooth: $t\in(0,2]$, $z_t=2-t$, $c_t=\alpha_t(1-\cos(\beta_tz_t))-z_t^t$,
  and
  \begin{equation}
  \theta_t(z) = \alpha_t(1-\cos(\beta_tz)),\quad {\rm for\ }z\in[0,z_t].
                \label{eq:smoothet}
  \end{equation}
  Here, $\beta_tz_t$ is the unique solution in the interval $(0,\pi)$ of the equation
  \begin{equation}
  (t-1)\tan(\beta_tz_t) = \beta_tz_t, \label{eq:taneqn}
  \end{equation}
  and $\alpha_t\beta_t\sin(\beta_tz_t)=tz_t^{t-1}$.  (If $t=1$ then
  $\beta_t=\beta_tz_t=\pi/2$.)  The compound function
  $\R\ni z\mapsto\theta_t(|z|)\in\R$ is then of class $C^2$.
\end{enumerate}
\end{example}

For some $d\in\N$, let $\cx$ be the $\sigma$-algebra of Lebesgue measurable subsets
of $\R^d$, and let $\mu_t$ be the following product probability measure on
$(\R^d,\cx)$:
\begin{equation}
\mu_t(dx) = r_t(x)\,dx := \exp(l_t(x))dx, \ {\rm where\ \ }
            l_t(x):=\textstyle\sum_i (C_t-\theta_t(|x_i|)), \label{eq:rdef}
\end{equation}
and $C_t\in\R$ is such that $\int \exp(C_t-\theta_t(|z|))dz=1$.  In what follows, we
shall suppress the subscript $t$, and so $l_t$, $r_t$ and $\mu_t$ will become $l$,
$r$ and $\mu$, etc.

For any $1\le\lambda<\infty$, let $L^\lambda(\mu)$ be the Banach space of
(equivalence classes of) measurable functions $u:\R^d\rightarrow\R$ for which
$\|u\|_{L^\lambda(\mu)}:=(\int |u|^\lambda d\mu)^{1/\lambda}<\infty$. Let 
$C^\infty(\R^d;\R)$ be the space of continuous functions with continuous partial
derivatives of all orders, and let $C_0^\infty(\R^d;\R)$ be the subspace of
those functions having compact support.

For $k\in\N$, let $S:=\{0,\ldots, k\}^d$ be the set of $d$-tuples of integers in
the range $0\le s_i\le k$.  For $s\in S$, we define $|s|=\sum_is_i$, and denote by
$0$ the $d$-tuple for which $|s|=0$.  For any $0\le j\le k$,
$S_j:=\{s\in S:j\le |s|\le k\}$ is the set of $d$-tuples of weight at least $j$ and
at most $k$.  Let $\Lambda=(\lambda_0,\lambda_1,\ldots,\lambda_k)$, where
$1\le\lambda_k\le\lambda_{k-1}\le\cdots\le\lambda_0<\infty$,
and let $W^{k,\Lambda}(\mu)$ be the mixed-norm, weighted Sobolev
space comprising functions $a\in L^{\lambda_0}(\mu)$ that have weak partial
derivatives $D^sa\in L^{\lambda_{|s|}}(\mu)$, for all $s\in S_1$.  For
$a\in W^{k,\Lambda}(\mu)$ we define
\begin{equation} 
\|a\|_{W^{k,\Lambda}(\mu)}
  := \bigg(\sum_{s\in S_0}\|D^sa\|_{L^{\lambda_{|s|}}(\mu)}^{\lambda_0}
     \bigg)^{1/\lambda_0} < \infty.  \label{eq:asnorm}
\end{equation}
The following theorem is a variant of a standard result in the theory of fixed-norm,
unweighted Sobolev spaces.

\begin{theorem} \label{th:banspa}
The space $W^{k,\Lambda}(\mu)$ is a Banach space.
\end{theorem}

\begin{proof}
That $\|\cdot\|_{W^{k,\Lambda}(\mu)}$ satisfies the axioms of a norm is easily
verified.  Suppose that $(a_n\in W^{k,\Lambda}(\mu))$ is a Cauchy sequence in this
norm; then, since the spaces $L^{\lambda_j}(\mu),\,0\le j\le k$ are all complete,
there exist functions $v_s\in L^{\lambda_{|s|}}(\mu), s\in S_0$ such that
$D^sa_n \rightarrow v_s {\rm\ in\ }L^{\lambda_{|s|}}(\mu)$. For any $s\in S_0$, and
any $\varphi\in C_0^\infty(\R^d;\R)$,
\begin{eqnarray}
\left|\int (D^sa_n-v_s)\varphi\,dx\right|
  & \le & \int |D^sa_n-v_s||\varphi|\,dx \nonumber \\
  & =  & \int |D^sa_n-v_s||\varphi|r^{-1}\mu(dx) \label{eq:weakder} \\
  & \le & \sup_{x\in\supp(\varphi)}(|\varphi|/r) \|D^sa_n-v_s\|_{L^1(\mu)}
          \rightarrow 0, \nonumber
\end{eqnarray}
and so
\[
\int v_s\varphi\,dx
  = \lim_n\int D^sa_n\varphi\,dx
  = (-1)^{|s|}\lim_n\int a_nD^s\varphi\,dx
  = (-1)^{|s|}\int v_0D^s\varphi\,dx,
\]
$v_0$ admits weak derivatives up to order $k$, and $D^sv_0=v_s$.  So
$W^{k,\Lambda}(\mu)$ is complete.
\end{proof}

The following developments show that functions in $W^{k,\Lambda}(\mu)$ can be
approximated by particular functions in $C^\infty(\R^d;\R)$ or $C_0^\infty(\R^d;\R)$.
For any $z\in(0,\infty)$, let $B_z:=\{x\in\R^d: |x|\le z\}$.  Let
$J\in C_0^\infty(\R^d;[0,\infty))$ be a function having the following properties:
(i) $\supp(J)=B_1$; (ii) $\int J\,dx = 1$.  For any $0<\epsilon<1$, let
$\Jeps(x)=\epsilon^{-d}J(x/\epsilon)$; then $\Jeps$ also has unit integral, but is
supported on $B_\epsilon$.  Since $l$ is bounded on bounded sets, any $u\in L^1(\mu)$
is also in $L_{{\rm loc}}^1(dx)$, and we can define the {\em mollified} version
$J_\epsilon\ast u\in C^\infty(\R^d;\R)$ as follows:
\begin{equation}
(\Jeps\ast u)(x)  := \int\Jeps(x-y)u(y)\,dy.   \label{eq:molver}
\end{equation}
For any $m\in\N$, let $\cu_m\subset L^1(\mu)$ comprise those functions
that take the value zero on the complement of $B_m$.  If $u\in\cu_m$ then
$J_\epsilon\ast u\in C_0^\infty(B_{m+1};\R)$.

\begin{lem} \label{le:molprop}
\begin{enumerate}
\item[(i)] For any $\lambda\in[1,\infty)$ and any $u\in\cu_m\cap L^\lambda(\mu)$,
  there exists an $\epsilon>0$ such that
  \begin{equation}
  \|\Jeps\ast u-u\|_{L^\lambda(\mu)}<1/m. \label{eq:Jumbnd}
  \end{equation}
\item[(ii)] For any $a\in W^{k,\Lambda}(\mu)$, $\epsilon>0$ and $s\in S_1$,
  $D^s(\Jeps\ast a)=\Jeps\ast (D^sa)$.
\end{enumerate}
\end{lem}

\begin{proof}
It follows from Jensen's inequality that, for any $\lambda\in[1,\infty)$,
\[
|(\Jeps\ast u)(x)|^\lambda \le (\Jeps\ast|u|^\lambda)(x)
  =   \int\Jeps(x-y)|u(y)|^\lambda r(y)^{-1}\mu(dy).
\]
Since $l$ is uniformly continuous on $B_{2m+1}$, there exists an $\alpha_m>0$
such that $|l(x)-l(y)|\le\lambda\log 2$ for all $y\in B_{2m}$, $|x-y|\le\alpha_m$.
So, for any $0<\epsilon<\alpha_m$,
\begin{eqnarray}
\|\Jeps\ast u\|_{L^\lambda(\mu)}^\lambda
  & \le & \int\int \Jeps(x-y)|u(y)|^\lambda\exp(l(x)-l(y))\mu(dy)dx \nonumber \\
          [-1.5ex] \label{eq:molbnd} \\ [-1.5ex]
  & \le & 2^\lambda\int\int \Jeps(x-y)dx|u(y)|^\lambda\mu(dy)     
          = 2^\lambda\|u\|_{L^\lambda(\mu)}^\lambda. \nonumber
\end{eqnarray}
It is a standard result that there exists a $\varphi\in C_0^\infty(B_{2m};\R)$ such
that $\|u-\varphi\|_{L^\lambda(\mu)}<1/6m$, which together with (\ref{eq:molbnd})
shows that, for any $0<\epsilon<\alpha_m$,
$\|\Jeps\ast u-\Jeps\ast\varphi\|_{L^\lambda(\mu)}<1/3m$.
Furthermore,
\[
|(\Jeps\ast\varphi)(x)-\varphi(x)|
  \le \int\Jeps(x-y)|\varphi(y)-\varphi(x)|\,dy 
  \le \sup_{|x-y|\le\epsilon}|\varphi(y)-\varphi(x)|.
\]
Since $\varphi$ is uniformly continuous, there exists a $\beta_u>0$ such that, for
any $0<\epsilon<\beta_u$ and all $x$, $|(\Jeps\ast\varphi)(x)-\varphi(x)| < 1/3m$.
We can now choose $0<\epsilon<\min\{\alpha_m,\beta_u\}$, which completes the
proof of part (i).

For $a$ and $s$ as in part (ii), and any $\varphi\in C_0^\infty(\R^d;\R)$,
\begin{eqnarray*}
\int (\Jeps\ast a)(x)D^s\varphi(x)\,dx
  & = & \int\int\Jeps(y)a(x-y)\,dyD^s\varphi(x)\,dx \\
  & = & \int\int a(x-y)D^s\varphi(x)\,dx \Jeps(y)\,dy \\
  & = & (-1)^{|s|}\int\int D^sa(x-y)\varphi(x)\,dx\Jeps(y)\,dy \\
  & = & (-1)^{|s|}\int(\Jeps\ast D^sa)(x)\varphi(x)\,dx,
\end{eqnarray*}
where we have used integration by parts $|s|$ times in the third step.  This
completes the proof of part (ii).
\end{proof}

For ease of notation in what follows, we shall abbreviate $\Jeps\ast u$ to $Ju$,
where it is understood that $\epsilon$ has been chosen as in part (i).  With
this convention, we can express part (ii) as $D^s(Ja)=J(D^sa)$, where it is
understood that $\epsilon$ has been chosen to satisfy (\ref{eq:Jumbnd}) for both
$a$ and $D^sa$.

For any $a\in W^{k,\Lambda}(\mu)$ and $m\in\N$, let $a_m(x):=a(x)\rho(x/m)$, where
$\rho\in C_0^\infty(\R^d;\R)$ is such that
\begin{equation}
\rho(x) = 1 {\rm \ \ if\ }|x|\le 1/2 \quad{\rm and}\quad
\rho(x) = 0 {\rm \ \ if\ }|x|\ge 1. \label{eq:rhodef}
\end{equation}

\begin{lem} \label{le:monoleb}
  $Ja_m\rightarrow a$ in $W^{k,\Lambda}(\mu)$,
  and so $C_0^\infty(\R^d;\R)$
  is dense in $W^{k,\Lambda}(\mu)$.
\end{lem}

\begin{proof}
Since $S_0$ is finite we may choose $\epsilon>0$ such that
(\ref{eq:Jumbnd}) is satisfied for all $u=D^sa_m$ with $s\in S_0$ and
$\lambda=\lambda_{|s|}$.  According to the Leibniz rule,
\begin{equation}
D^s a_m
  = \sum_{\sigma\le s} m^{-|s-\sigma|} D^\sigma aD^{s-\sigma}\rho 
    \prod_{1\le i\le d}\left(\begin{array}{c} s_i \\ \sigma_i\end{array}\right),
    \label{eq:Dsfrho}
\end{equation}
and so $|D^s a_m|\le K \sum_\sigma|D^\sigma a|\in L^{\lambda_{|s|}}(\mu)$.
Since $D^sa_m\rightarrow D^sa$ for all $x$, it follows from the dominated convergence
theorem that it also converges in $L^{\lambda_{|s|}}(\mu)$.
Lemma \ref{le:molprop} completes the proof.
\end{proof}

\begin{remark} \label{re:hilbert}
If $\lambda_j=2$ for $0\le j\le k$ then $H^k(\mu):=W^{k,\Lambda}(\mu)$ is a Hilbert
Sobolev space with inner product
\begin{equation}
\langle a, b\rangle_H
  = \sum_{s\in S_0}\langle D^sa, D^sb\rangle_{L^2(\mu)}. \label{eq:hildot}
\end{equation}
\end{remark}

\section{The Manifolds of Finite Measures} \label{se:manifmt}

In this section, we construct manifolds of finite measures on $(\R^d,\cx)$ modelled
on the Sobolev spaces of section \ref{se:modspa}.  The charts of the manifolds are
based on the ``deformed logarithm'' $\log_d:(0,\infty)\rightarrow\R$, defined by
\begin{equation}
\log_dy = y-1 + \log y.  \label{eq:deflog}
\end{equation}
Now $\inf_y\log_d y=-\infty$, $\sup_y\log_d y=+\infty$, and
$\log_d\in C^\infty((0,\infty);\R)$ with strictly positive first derivative
$1+y^{-1}$, and so, according to the inverse function theorem,
$\log_d$ is a diffeomorphism from $(0,\infty)$ onto $\R$.  Let $\psi$ be its inverse.
This can be thought of as a ``deformed exponential'' function \cite{naud1}.  We
use $\psi^{(n)}$ to denote its $n$-th derivative and, for convenience, set
$\psi^{(0)}:=\psi$.

\begin{lem} \label{le:psidif}
\begin{enumerate}
\item[(i)] For any $n\in\N$:
  \begin{equation}
  (1+\psi)\psi^{(n)}
    = \psi^{(n-1)} - \half\sum_{j=1}^{n-1}
        \left(\begin{array}{c}n \\ j\end{array}\right)\psi^{(j)}\psi^{(n-j)};
        \label{eq:psirec}
\end{equation}
in particular $\psi^{(1)}=\psi/(1+\psi)>0$ and $\psi^{(2)}=\psi/(1+\psi)^3>0$, and
so $\psi$ is strictly increasing and convex.
\item[(ii)] For any $n\ge 2$,
\begin{equation}
\psi^{(n)} = \frac{Q_{n-2}(\psi)}{(1+\psi)^{2(n-1)}}\psi^{(1)},  \label{eq:psidrec}
\end{equation}
where $Q_{n-2}$ is a polynomial of degree no more than $n-2$.  In particular,
$\psi^{(n)}$, $\psi^{(n)}/\psi$ and $\psi^{(n)}/\psi^{(1)}$ are all bounded.
\end{enumerate}
\end{lem}

\begin{proof}
That $\psi^{(1)}$ and $\psi^{(2)}$ are as stated is verified by a straightforward
computation.  Both (\ref{eq:psirec}) and (\ref{eq:psidrec}) then follow by induction
arguments.
\end{proof}

Let $G:=W^{k,\Lambda}(\mu)$ be the general mixed-norm space of section \ref{se:modspa},
and let $M$ be the set of finite measures on $(\R^d,\cx)$ satisfying the following:
\begin{enumerate}
\item[(M1)] $P$ is mutually absolutely continuous with respect to $\mu$;
\item[(M2)] $\log_d p\in G$;
\end{enumerate}
(We denote measures on $(\R^d,\cx)$ by the upper-case letters $P$, $Q$, \ldots,
and their densities with respect to $\mu$ by the corresponding lower case letters,
$p$, $q$, \ldots)  In order to control both the density $p$ and its log, we employ
the ``balanced'' chart of \cite{newt4} and \cite{newt6},
$\phi:M\rightarrow G$.  This is defined by:
\begin{equation}
\phi(P) = \log_dp = p-1 + \log p.  \label{eq:phitdef}
\end{equation}

\begin{prop} \label{pr:biject}
$\phi$ is a bijection onto $G$.
\end{prop}

\begin{proof}
It follows from (M2) that, for any $P\in M$, $\phi(P)\in G$. Suppose, conversely,
that $a\in G$; then since $\psi^{(1)}$ is bounded, $\psi(a)\in L^1(\mu)$,
and so defines a finite measure $P(dx)=\psi(a(x))\mu(dx)$.  Since $\psi$ is strictly
positive, $P$ satisfies (M1).  That it also satisfies (M2) follows from the fact that
$\log_d\psi(a)=a\in G$.  We have thus shown that $P\in M$ and clearly $\phi(P)=a$.
\end{proof}

The inverse map $\phi^{-1}:G\rightarrow M$ takes the form
\begin{equation}
p(x) = \frac{d\phi^{-1}(a)}{d\mu}(x) = \psi(a(x)). \label{eq:phitinv}
\end{equation}
In \cite{newt4,newt6}, tangent vectors were defined as equivalence classes of
differentiable curves passing through a given base point, and having the same first
derivative at this point.  This allowed them to be interpreted as linear operators
acting on differentiable maps.  Here, we use a different definition that is closer
to that of membership of $M$. For any $P\in M$, let $\Ptil_a$ be the finite measure
on $(\R^d,\cx)$ with density $\ptil_a=\psi^{(1)}(a)$, where $a=\phi(P)$. 
($\Ptil_a\ll\mu$ since $\psi^{(1)}$ is bounded.)  We define a tangent vector $U$
at $P$ to be a {\em signed} measure on $(\R^d,\cx)$ that is absolutely continuous
with respect to $\Ptil_a$, with Radon-Nikodym derivative $dU/d\Ptil_a\in G$.  The
tangent space at $P$ is the linear space of all such measures, and the tangent bundle
is the disjoint union $TM:=\cup_{P\in M}(P,T_PM)$.  This is globally trivialised by
the chart
$\Phi:TM\rightarrow G\times G$, where
\begin{equation}
\Phi(P,U) = (\phi(P),dU/d\Ptil_a). \label{eq:bunchartt}
\end{equation}

The derivative of a (Fr\'{e}chet) differentiable, Banach-space-valued map
$f:M\rightarrow \bY$ (at $P$ and in the ``direction'' $U$) is defined in the obvious
way:
\begin{equation}
Uf = (f\circ\phi^{-1})_a^{(1)}u, \quad{\rm where\ }(a,u)=\Phi(P,U). \label{eq:tanvec}
\end{equation}
Clearly $u=U\phi$.  We shall also need a weaker notion of differentiability due to
Leslie \cite{lesl1,lesl2}.  Let $A:G\rightarrow Y$ be a continuous linear map
and, for fixed $a=\phi(P)\in G$, let $R:\R\times G\rightarrow Y$ be defined by
\[
R(y,u)
  = \left\{\begin{array}{ll}
      y^{-1}\left(f\circ\phi^{-1}(a+yu)-f\circ\phi^{-1}(a)\right)
      - Au & {\rm if\ }y\neq 0, \\
      0 & {\rm if\ }y=0. \end{array} \right.
\]
If $R$ is continuous at $(0,u)$ for all $u\in G$, then we say that $f$
is {\em Leslie differentiable} at $P$, with derivative
\begin{equation}
Uf = d(f\circ\phi^{-1})_au = Au. \label{eq:lesdif}
\end{equation}
If $f$ is Leslie differentiable at all $P\in M$ then we say that it is
Leslie differentiable.  This is a slightly stronger property than the
``$d$-differentiability'' used in \cite{newt4}, which essentially demands continuity
of $R$ in the first argument only.

The construction above defines an
infinite-dimensional manifold of finite measures, $(M,G,\phi)$, with atlas comprising
the single chart $\phi$.  $M$ is a subset of an instance of the manifold constructed
in \cite{newt6} (that in which the measurable space $\bX$ of \cite{newt6} is $\R^d$),
but has a stronger topology than the associated relative topology.  Results in
\cite{newt6} concerning the smoothness of maps defined on the model space
$L^{\lambda_0}(\mu)$ are true {\em a-forteriori} when the latter is replaced by $G$;
in fact, stronger results can be obtained under the following hypothesis:
\begin{enumerate}
\item[(E1)] $t\in(1,2]$ and $\lambda_1=\lambda_0$.
\end{enumerate}

For some $1\le\beta\le\lambda_0$, let $\Psi_\beta:G\rightarrow L^\beta(\mu)$ be the
nonlinear superposition (Nemytskii) operator defined by $\Psi_\beta(a)(x)=\psi(a(x))$.
(See \cite{apza1}.)

\begin{lem} \label{le:superp}
\begin{enumerate}
\item[(i)] $\Psi_\beta\in C^N(G;L^\beta(\mu))$, where
  \begin{equation}
  N = N(\lambda_0,\lambda_1,\beta,t)
    := \left\{\begin{array}{ll}
        \lceil\lambda_0/\beta\rceil-1 & {\it if\ (E1)\ does\ not\ hold,}  \\
        \lfloor\lambda_0/\beta\rfloor & {\it if\ (E1)\ holds.}\end{array}\right.
        \label{eq:Ndef}
  \end{equation}
  For any $1\le j\le N$, $\Psi_\beta^{(j)}:G\rightarrow L(G^j;\,L^\beta(\mu))$ is
  as follows
  \begin{equation}
  \Psi_{\beta,a}^{(j)}(u_1,\ldots,u_j)(x) = \psi^{(j)}(a(x))u_1(x)\cdots u_j(x).
     \label{eq:Xibetder1}
  \end{equation}
\item[(ii)]  If $\lambda_0/\beta\in\N$ and (E1) does not hold, then the highest
  Fr\'{e}chet derivative, $\Psi_\beta^{(N)}$, is Leslie differentiable, with
  derivative
  \begin{equation}
  (d\Psi_{\beta,a}^{(N)}u_{N+1})(u_1,\ldots,u_N)(x)
     = \psi^{(N+1)}(a(x))u_1(x)\cdots u_{N+1}(x). \label{eq:Xibetder2}
  \end{equation}
\item[(iii)] $\Psi_\beta$ satisfies global Lipschitz continuity and linear growth
  conditions, and all its derivatives (including that in (\ref{eq:Xibetder2})) are
  globally bounded.
\end{enumerate}
\end{lem}

\begin{proof}
According to the mean value theorem, for any $a,b\in G$,
\begin{equation}
\psi(b)-\psi(a) = \psi^{(1)}(\alpha b+(1-\alpha)a)(b-a)
  \quad{\rm for\ some\ }0\le\alpha(x)\le 1, \label{eq:psiba}
\end{equation}
and so the Lipschitz continuity and linear growth of $\Psi_\beta$ follow from the
boundedness of $\psi^{(1)}$.
Let $(a_n\in G\setminus\{a\})$ be a sequence converging to $a$ in $G$.
For any $1\le j\le N$ let
\begin{eqnarray}
\Delta_n
  & := & \psi^{(j-1)}(a_n) - \psi^{(j-1)}(a) - \psi^{(j)}(a)(a_n-a) 
         \label{eq:Deldef}\\
\Gamma_n
  & := & \psi^{(j)}(a_n)-\psi^{(j)}(a). \nonumber
\end{eqnarray}
According to the mean-value theorem $\Delta_n=\delta_n(a_n-a)$, where
\[
\delta_n = \psi^{(j)}(\alpha_na_n+(1-\alpha_n)a) - \psi^{(j)}(a)
            \quad{\rm for\ some\ }0\le\alpha_n(x)\le 1.
\]
H\"{o}lder's inequality shows that, for all $u_1,\ldots,u_j$ in the unit ball
of $G$,
\[
\|\Delta_n u_1\cdots u_{j-1}\|_{L^\beta(\mu)} \le \|\Delta_n\|_{L^\nu(\mu)}
  \quad{\rm and}\quad
\|\Gamma_n u_1\cdots u_j\|_{L^\beta(\mu)} \le \|\Gamma_n u_j\|_{L^\nu(\mu)},
\]
where $\nu:=\lambda_0\beta/(\lambda_0-(N-1)\beta)$.  In order to prove part (i),
it thus suffices to show that
\begin{equation}
\|a_n-a\|_G^{-1}\|\Delta_n\|_{L^\nu(\mu)}\rightarrow 0 \quad{\rm and}\quad
  \sup_{\|u\|_G=1}\|\Gamma_n u\|_{L^\nu(\mu)}\rightarrow 0. \label{eq:Delcon}
\end{equation}

If $\nu<\lambda_0$ (eg.~if (E1) does not hold) then H\"{o}lder's inequality
shows that
\[
\|\Delta_n\|_{L^\nu(\mu)}
  \le \|\delta_n\|_{L^\zeta(\mu)}\|a_n-a\|_{L^{\lambda_0}(\mu)}
      \quad{\rm and}\quad
\|\Gamma_nu\|_{L^\nu(\mu)}
  \le \|\Gamma_n\|_{L^\zeta(\mu)}\|u\|_{L^{\lambda_0}(\mu)},
\]
where $\zeta:=\lambda_0\nu/(\lambda_0-\nu)$.
Now $\delta_n$ and $\Gamma_n$ are bounded and converge to zero in probability, and
so the bounded convergence theorem establishes (\ref{eq:Delcon}).

If $\nu=\lambda_0$ then (E1) holds.  Suppose first that $\nu>1$, and let
$f_m\in C^\infty(\R^d;\R)$ be a sequence converging in $G$ to some $b\in G$.
For some $1\le i\le d$ and a weakly differentiable $g:\R^d\rightarrow\R$, let
$g^\prime:=\partial g/\partial x_i$; then
\[
\|(|f_m|^\nu)^\prime - h(b)b^\prime\|_{L^1(\mu)} = R_m + T_m,
\]
where $h\in C(\R;\R)$ is defined by $h(y)=\nu |y|^{\nu-1}{\rm sgn}(y)$,
\begin{eqnarray*}
R_m & := & \|(h(f_m)-h(b))f_m^\prime\|_{L^1(\mu)}
           \le  K\|h(f_m)-h(b)\|_{L^{\nu^*}(\mu)}, \\
T_m & := & \|h(b)(f_m^\prime-b^\prime)\|_{L^1(\mu)}
           \le  K\|f_m^\prime-b^\prime\|_{L^\nu(\mu)},
\end{eqnarray*}
$\nu^*:=\nu/(\nu-1)$, $K<\infty$ and we have used H\"{o}lder's inequality in the
bounds on $R_m$ and $T_m$.  With a slight abuse of notation, let $f_m$ be a
subsequence that converges to $b$ almost surely; then $h(f_m)\rightarrow h(b)$
almost surely and $\|h(f_m)\|_{L^{\nu^*}(\mu)}\rightarrow\|h(b)\|_{L^{\nu^*}(\mu)}$.
So it follows from Proposition 4.7.30 in \cite{boga1} that $R_m\rightarrow 0$.
Clearly $T_m\rightarrow 0$.  As in the proof of Theorem \ref{th:banspa}, this
shows that $|b|^\nu$ is weakly differentiable with respect to $x_i$, with derivative
\begin{equation}
(|b|^\nu)^\prime = h(b)b^\prime \in L^1(\mu).  \label{eq:dercon}
\end{equation}
This enables the use of a log-Sobolev inequality.
Let $\alpha:=(t-1)/t$, and let $F_\alpha, G_\alpha:[0,\infty)\rightarrow[0,\infty)$
be the complementary {\em Young functions} defined by
\begin{equation}
F_\alpha(z) = \int_0^z\log^\alpha(y+1)\;dy \quad{\rm and}\quad
G_\alpha(z) = \int_0^z\left(\exp(y^{1/\alpha})-1\right)\;dy. \label{eq:Youngdef}
\end{equation}
(See, for example, \cite{rare1}.)  $F_\alpha$ is equivalent to any Young function
$\Ftil_\alpha$, for which $\Ftil_\alpha(z)=z\log^\alpha z$ for $z\ge 2$, in the sense
that there exist constants $0<c_1<c_2<\infty$ such that, for all sufficiently large
$z$, $F_\alpha(c_1z) \le \Ftil_\alpha(z) \le F_\alpha(c_2z)$. Similarly, $G_\alpha$
is equivalent to any Young function $\tilde{G}_\alpha$, for which
$\tilde{G}_\alpha(z)=\exp(z^{1/\alpha})$ for $z\ge 2$.  We denote the
associated Orlicz spaces $L^1\log^\alpha L(\mu)$ and $\exp L^{1/\alpha}(\mu)$,
respectively.  $L^1\log^\alpha L(\mu)$ is equal (modulo equivalent norms) to the
Lorentz-Zygmund space $L^{1,1;\alpha}(\mu)$, which in the context of
the product probability space $(\R^d,\cx,\mu)$ is a {\em rearrangement-invariant
space}. (See section 3 in \cite{cipi1}.)  It follows from Theorem 7.12 in
\cite{cipi1}, together with (\ref{eq:dercon}), that
\[
\left\||b|^\nu\right\|_{L^1\log^\alpha L(\mu)} \le K\|b\|_G^\nu,
  \quad{\rm for\ some\ }K<\infty.
\]
This is clearly also true if $\nu=1$.
In the light of the generalised H\"{o}lder inequality, in order to prove
(\ref{eq:Delcon}) it now suffices to show that the sequences $|\delta_n|^\nu$ and
$|\Gamma_n|^\nu$ converge to zero in $\exp L^{1/\alpha}(\mu)$,
but this follows from their boundedness and convergence to zero in probability.
This completes the proof of part (i).

With the hypotheses of part (ii), let $(t_n\in\R\setminus\{0\})$ and $(v_n\in G)$
be sequences converging to $0$ and $u_{N+1}$, respectively, and let
$a_n:=a+t_nv_n$.  Substituting this sequence into (\ref{eq:Deldef}), we obtain
\[
t_n^{-1}\Delta_n = \delta_nv_n = \delta_n(v_n-u_{N+1}) + \delta_nu_{N+1}.
\]
Both terms on the right-hand side here converge to zero in $L^{\lambda_0}(\mu)$
since $\delta_n$ is bounded and converges to zero in probability.  This
completes the proof of part (ii).  Part (iii) follows from (\ref{eq:psiba}) and
the boundedness of the $\psi^{(j)}$.
\end{proof}

For $1\le\beta\le\lambda_0$, let $m_\beta,e_\beta:M\rightarrow L^\beta(\mu)$ be
defined by
\begin{equation}
m_\beta(P) = \Psi_\beta(\phi(P)) - 1 \quad{\rm and}\quad
e_\beta(P) = \imath\circ\phi(P) - m_\beta(P),\label{eq:metil}
\end{equation}
where $\imath:G\rightarrow L^\beta(\mu)$ is the inclusion map.  These are injective
and share the smoothness properties of $\Psi_\beta$ developed in Lemma
\ref{le:superp}.  In particular,
\begin{equation}
Um_\beta = \psi^{(1)}(a)\frac{dU}{d\Ptil_a} = \frac{dU}{d\mu}\qquad{\rm and} \qquad
Ue_\beta = \frac{\psi^{(1)}}{\psi}(a)\frac{dU}{d\Ptil_a} = \frac{dU}{dP}, \label{eq:meder}
\end{equation}
where $a=\phi(P)$, and the derivatives are Leslie derivatives if $\beta=\lambda_0$,
and (E1) does not hold.  The maps $m_\beta$ and $e_\beta$ can be used to investigate the regularity of statistical divergences on $M$.    The usual extension of the KL
divergence to sets of finite measures, such as $M$, is
\cite{amar1}:
\begin{eqnarray}
\cd(P\,|\,Q)
  & = & Q(\R^d) - P(\R^d) + \Emu p\log(p/q) \nonumber \\
        [-1.5ex] \label{eq:KLdiv} \\ [-1.5ex]
  & = & \Emu m_1(Q)-\Emu m_1(P) + \Emu (m_2(P)+1)(e_2(P)-e_2(Q)),\quad \nonumber
\end{eqnarray}
where $\Emu$ is expectation (integration) with respect to $\mu$.  This clearly
requires $\lambda_0\ge 2$.  Its smoothness is investigated in \cite{newt6}; $\cd$
admits mixed second partial derivatives (in the sense of Leslie if $\lambda_0=2$
and (E1) does not hold).  So we can use Eguchi's characterisation of the Fisher-Rao
metric on $T_PM$ \cite{eguc1}: for any $U,V\in T_PM$,
\begin{equation}
\langle U, V \rangle_P := -UV\cd
  = \langle Um_2,Ve_2\rangle_{L^2(\mu)}
  = \Emu\frac{p}{(1+p)^2}U\phi V\phi. \label{eq:fisher}
\end{equation}
It follows that  $\langle V, U \rangle_P = \langle U, V \rangle_P$ and that 
$\langle yU, V \rangle_P=\langle U, yV \rangle_P = y\langle U, V \rangle_P$
for any $y\in\R$; furthermore,
\begin{equation}
\langle U, U \rangle_P
  \le \Emu (U\phi)^2 \le \|U\phi\|_G^2, \label{eq:fishdom}
\end{equation}
and $\langle U,U\rangle_P=0$ if and only if $U\phi=0$.  So the metric is
positive definite and dominated by the chart-induced norm on $T_PM$.  However the
Fisher-Rao metric and chart-induced norm are not equivalent, even when the model
space is $L^2(\mu)$ \cite{newt4}.  In the general, infinite-dimensional case
$(T_PM,\langle\fndot,\fndot\rangle_P)$ is not a Hilbert space; the Fisher-Rao
metric is a {\em weak} Riemannian metric.

If $\lambda_0\ge 3$ then $M$ also admits the Amari-Chentsov tensor.  This is the
symmetric covariant 3-tensor field defined by
\begin{equation}
\tau_P(U,V,W) = \Emu Um_3Ve_3We_3 = \Emu\frac{p}{(1+p)^3}U\phi V\phi W\phi.
                \label{eq:AmChen}
\end{equation}
The regularity of the Fisher-Rao metric and higher-order covariant tensors can be
derived from that of $\Psi_\beta$, as developed in Lemma \ref{le:superp}.  They
become smoother with increasing values of $\lambda_0$.  Log-Sobolev embedding
enhances this gain for particular integer values of $\lambda_0$.  Suppose,
for example, that $\lambda_0=2$.  If (E1) holds then the metric is a continuous
covariant 2-tensor on $M$; however if (E1) does not hold then, although the
composite map $M\ni P\mapsto\langle\bfU(P),\bfV(P)\rangle_P\in\R$ is continuous for
all continuous vector fields $\bfU,\bfV$, the metric is not continuous in the sense
of the operator norm.

If $\lambda_0\ge 2$ the variables $m_2$ and $e_2$ are {\em bi-orthogonal}
representations of measures in $M$.  This can be seen in the following
{\em generalised cosine rule}:
\begin{eqnarray}
\cd(P|R)
  & = & \cd(P|Q) + \cd(Q|R) \nonumber \\
        [-1.5ex] \label{eq:cosine} \\ [-1.5ex]
  &   & \quad - \left\langle m_2(P)-m_2(Q), e_2(R)-e_2(Q)
          \right\rangle_{L^2(\mu)}. \nonumber
\end{eqnarray}
Setting $R=P$ and using the fact that $m_2+e_2=\imath\circ\phi$, where
$\imath:G\rightarrow L^2(\mu)$ is the inclusion map, we obtain the global bound
\begin{eqnarray}
\cd(P|Q) + \cd(Q|P)
   & = & \langle m_2(P)-m_2(Q)\,,\, e_2(P)-e_2(Q)\rangle_{L^2(\mu)} \nonumber \\
         [-1.5ex] \label{eq:divsumt} \\ [-1.5ex]
   & \le & \half\|\phi(P)-\phi(Q)\|_{L^2(\mu)}^2
           \le \half\|\phi(P)-\phi(Q)\|_G^2. \nonumber
\end{eqnarray}

\section{Special Model Spaces} \label{se:specms}

The construction of $M$ and $TM$ in the previous section is valid for any of the
weighted mixed-norm spaces developed in section \ref{se:modspa}, including
the {\em fixed norm} space $G_f:=W^{k,(\lambda,\ldots,\lambda)}(\mu)$. However,
certain spaces are particularly suited to the deformed exponential function $\psi$;
these are introduced next.  A special class of mixed-norm spaces, on which the
nonlinear superposition operators associated with $\psi$ {\em act continuously}, is
developed in section \ref{se:gtil}.  Section \ref{se:fixnor} investigates fixed-norm
spaces and shows that, with the exception of the cases $k=1$, $\lambda\in[1,\infty)$
and $k=2$, $\lambda=1$, they do not share this property.

\subsection{A Family of Mixed Norm Spaces} \label{se:gtil}

This section develops the mixed-norm space $\Gtil:=W^{k,\Lambda}(\mu)$ with
$\lambda_0\ge\lambda_1\ge k$ and $\lambda_j=\lambda_1/j$ for $2\le j\le k$.
Lemma \ref{le:superp} can be augmented as follows.

\begin{prop} \label{pr:rapsi}
\begin{enumerate}
\item[(i)] For any $a\in \Gtil$, $\psi(a)\in \Gtil$.
\item[(ii)] The nonlinear superposition (Nemytskii) operator
  $\Psi_m:\Gtil\rightarrow\Gtil$, defined by $\Psi_m(a)(x)=\psi(a(x))$, is continuous.
\item[(iii)] $\Psi_m(\Gtil)$ is convex.
\end{enumerate}
\end{prop}

\begin{proof}
A {\em partition} of $s\in S_1$ is a set $\pi=\{\sigma_1,\ldots,\sigma_n\in S_1\}$
such that $\sum_i\sigma_i=s$. Let $\Pi(s)$ denote the set of distinct partitions of
$s$ and, for any $\pi\in\Pi(s)$, let $|\pi|$ denote the number of $d$-tuples in
$\pi$.  According to the Fa\'{a} di Bruno formula, for any $s\in S_1$ and any
$f\in C^\infty(\R^d;\R)$,
\begin{equation}
D^s\psi(f) = F_s(f)
  := \sum_{\pi\in\Pi(s)}K_\pi\psi^{(|\pi|)}(f)\prod_{\sigma\in\pi}D^\sigma f,
       \label{eq:FdiB}
\end{equation}
where the $K_\pi<\infty$ are combinatoric constants. $D^s\psi(f)\in C^\infty(\R^d;\R)$
since the derivatives of $\psi$ are bounded and $D^\sigma f\in C^\infty(\R^d;\R)$ for
all $\sigma\in\pi$.  We set $F_0:=\psi$, and extend the domain of $F_s$ to
$\Gtil$ in the obvious way.

Let $(f_n\in C^\infty(\R^d;\R))$ be a sequence converging in the sense of $\Gtil$ to
$a$. Since the first derivative of $\psi$ is bounded, the mean value
theorem shows that $\psi(f_n)\rightarrow \psi(a)=F_0(a)$ in the sense of
$L^{\lambda_0}(\mu)$.  Furthermore, for any $s\in S_1$,
\begin{eqnarray}
\left|D^s\psi(f_n)-F_s(a)\right|
  & \le & K\sum_{\pi\in\Pi(s)}\left|\psi^{(|\pi|)}(f_n)\right|\Gamma_{\pi,n}
          \nonumber \\
          [-1.5ex] \label{eq:serrspl} \\ [-1.5ex]
  &     & \quad + K\sum_{\pi\in\Pi(s)}
          \left|\psi^{(|\pi|)}(f_n)-\psi^{(|\pi|)}(a)\right|
            \prod_{\sigma\in\pi}|D^\sigma a|,\quad \nonumber
\end{eqnarray}
where
\[
\Gamma_{\pi,n}
  := \bigg|\prod_{\sigma\in\pi}D^\sigma f_n
          -\prod_{\sigma\in\pi}D^\sigma a\,\bigg| 
  \le  \sum_{\sigma\in\pi}|D^\sigma (f_n-a)|
          \prod_{\tau\in\pi\setminus\{\sigma\}} \left(|D^\tau f_n|+|D^\tau a|\right).
\]
Now $\sum_{\sigma\in\pi}|\sigma|=|s|$, and so it follows from H\"{o}lder's inequality that
\[
\|\Gamma_{\pi,n}\|_{L^{\lambda/|s|}(\mu)}
  \le \sum_{\sigma\in\pi}\|D^\sigma (f_n-a)\|_{L^{\lambda/|\sigma|}(\mu)}
      \prod_{\tau\in\pi\setminus\{\sigma\}}
      \left\||D^\tau f_n|+|D^\tau a|\right\|_{L^{\lambda/|\tau|}(\mu)},
\]
which, together with the boundedness of the derivatives of $\psi$, shows that the first
term on the right-hand side of (\ref{eq:serrspl}) converges to zero in the sense of
$L^{\lambda/|s|}(\mu)$.  The second term converges to zero in probability and is
dominated by the function $C\prod_{\sigma\in\pi}|D^\sigma a|\in L^{\lambda/|s|}(\mu)$
for some $C<\infty$, and so it also converges to zero in the sense of
$L^{\lambda/|s|}(\mu)$.  We have thus shown that, for any $s\in S_0$, $D^s\psi(f_n)$
converges to $F_s(a)$ in the sense of $L^{\lambda/|s|}(\mu)$.  In particular,
$F_s(a)\in L^{\lambda/|s|}(\mu)$.  That $\psi(a)$ is weakly differentiable with
derivatives $D^s\psi(a)=F_s(a)$, for all $s\in S_1$, follows from arguments similar
to those in (\ref{eq:weakder}) with $f_n$ playing the role of $a_n$, and this
completes the proof of part (i).

Let $(a_n\in G_m)$ be a sequence converging to $a$ in the sense of $\Gtil$.  The
above arguments, with $a_n$ replacing $f_n$, show that, for any $s\in S_0$,
$F_s(a_n)\rightarrow F_s(a)$ in the sense of $L^{\lambda/|s|}(\mu)$,
and this completes the proof of part (ii).

For any $P_0,P_1\in M$ and any $y\in(0,1)$, let $P_y:=(1-y)P_0+yP_1$.  Clearly
$p_y\in\Gtil$; we must show that $\log p_y\in\Gtil$.  Let
$f:(0,\infty)\rightarrow\R$ be defined by
\[
f(z) = \indic_{(0,1)}(z)(-\log z)^\lambda + \indic_{[1,\infty)}(z)(z-1)^\lambda;
\]
then $|\log z|^\lambda\le f(z)$, and $f$ is of class $C^2$ with non-negative second
derivative, and so is convex.  It follows from Jensen's inequality that
\[
\Emu|\log p_y|^\lambda
   \le \Emu f(p_y) \le (1-y)\Emu f(p_0) + y\Emu f(p_1) < \infty. 
\]
A further application of the Fa\'{a} di Bruno formula shows that, for any $s\in S_1$,
\[
|D^s\log p_y|
  \le K_1\sum_{\pi\in\Pi(s)}
            |\log^{(|\pi|)}(p_y)|\prod_{\sigma\in\pi}|D^\sigma p_y|
  \le K_2\sum_{\pi\in\Pi(s)} \prod_{\sigma\in\pi}\left|\frac{D^\sigma p_0}{p_0}
            +\frac{D^\sigma p_1}{p_1}\right|.
\]
Now $p_i=\psi(a_i)$ for some $a_0,a_1\in\Gtil$, and so
$D^\sigma p_i/p_i=F_\sigma(a_i)/\psi(a_i)$.  Since $\psi^{(n)}/\psi$ is bounded, the
arguments used above to show that $D^s\psi(a_i)\in L^{\lambda/|s|}(\mu)$ can be used
to show that $D^\sigma\psi(a_i)/\psi(a_i)\in L^{\lambda/|\sigma|}(\mu)$.
H\"{o}lder's inequality then shows that $D^s\log p_y\in L^{\lambda/|s|}(\mu)$.  We
have thus shown that $\log p_y\in\Gtil$.  So $P_y\in M$, and this completes the
proof of part (iii).
\end{proof}

\subsection{Fixed Norm Spaces} \label{se:fixnor}

Proposition \ref{pr:rapsi} shows that the function $\psi$ defines a superposition
operator that ``acts continuously'' on the mixed norm Sobolev space $\Gtil$.  The
question naturally arises whether or not it has this property with respect to any
fixed norm spaces (other than $W^{1,1}(\mu)$).  Since, for $k\ge 2$ and
$\lambda\ge\lambda_0$, the space $G_f=W^{k,(\lambda,\ldots,\lambda)}(\mu)$ is a
subset of $\Gtil$ and has a topology stronger than the relative topology, it is
clear that $\psi(G_f)\subset\psi(\Gtil)\subset\Gtil$, and that the restriction, 
$\Psi:G_f\rightarrow\Gtil$, is continuous. However, except in one specific case,
$\psi(G_f)$ is not a subset of $G_f$, as the following proposition shows.

\begin{prop} \label{pr:fixedact} 
If $\lambda>1$ and $k\ge 2$ then there exists an $a\in G_f$ for which
$\psi(a)\notin G_f$.
\end{prop}

\begin{proof} (Adapted from Dahlberg's counterexample.) Let $t\in(0,2]$, $z_t\ge 0$
and $l_t:\R^d\rightarrow\R$ be as in section \ref{se:modspa}, and let
$\{S_n\subset\R^d, n\in\N\}$ be the sequence of closed spheres with centres
$\sigma_n=(n^{1/t},0,\ldots,0)$ and radii $1/n$.  If $x\in S_n$ then
$|l_t(x)-l_t(\sigma_n)|\le K/\sqrt{n}$ for some $K<\infty$. Let
$\varphi\in C_0^\infty(\R^d;\R)$ be such that
\[
\varphi(y) = y_1 \quad{\rm if\ \ }|y|\le 1/2\quad{\rm and}\quad
  \varphi(y) = 0 \quad{\rm if\ \ }|y|\ge 1.
\]
Since $\psi$ is not a polynomial, its $k$'th derivative $\psi^{(k)}$ is not
identically zero, and we can choose $-\infty<\zeta_1<\zeta_2<\zeta_1+1$ such that
$|\psi^{(k)}(z)|\ge\epsilon$ for all $z\in[\zeta_1,\zeta_2]$ and some $\epsilon>0$.
Finally, let $a:\R^d\rightarrow\R$ be defined by the sum
\begin{equation}
a(x) = \zeta_1 + \sum_{n=m}^\infty \alpha^n\varphi(n(x-\sigma_n)), \label{eq:alphandef}
\end{equation}
where $\alpha=\exp(2/((k+1)\lambda-1))$ and $m>z_t+1$. (The support of the $n$'th
term in the sum here is a subset of $S_n$, and so $a$ is well defined and of class
$C^\infty$.)  We claim that $a\in G_f$; in fact, for any $s\in S_1$ with $|s|=j$,
\begin{eqnarray}
\Emu|D^sa|^\lambda
  & \le & K\sum_{n=m}^\infty \alpha^{\lambda n}n^{j\lambda}\exp(-n)
            \int|D^s\varphi(n(x-\sigma_n))|^\lambda dx \nonumber \\
             [-1.5ex] \label{eq:ldera} \\ [-1.5ex]
  &  =  & K\sum_{n=m}^\infty \alpha^{\lambda n}n^{j\lambda-d}\exp(-n)
            \int|D^s\varphi(y)|^\lambda dy < \infty, \nonumber
\end{eqnarray}
and a similar bound can be found for $\Emu|a-\zeta_1|^\lambda$.  It now suffices to
show that $D^s\psi(a)\notin L^\lambda(\mu)$, where $s=(k,0,\ldots,0)$. Let
\[
T_n := \left\{ x\in\R^d: |x-\sigma_n|\le 1/2n {\rm\ \ and\ \ }
       0 \le (x-\sigma_n)_1 \le (n\alpha^n)^{-1}(\zeta_2-\zeta_1)\right\};
\]
then, for any $x\in T_n$, $a(x)=\zeta_1+n\alpha^n(x-\sigma_n)_1\in[\zeta_1,\zeta_2]$,
and so
\begin{eqnarray}
\Emu|D^s\psi(a)|^\lambda
  & \ge & \sum_{n=m}^\infty \alpha^{k\lambda n}n^{k\lambda}
          \int_{T_n}|\psi^{(k)}(\zeta_1+n\alpha^n(x-\sigma_n)_1)|^\lambda r(x)\,dx
          \nonumber \\
  & \ge & K_1\epsilon^\lambda\sum_{n=m}^\infty \alpha^{k\lambda n}n^{k\lambda}
            \exp(-n) |T_n| \label{eq:kderpa} \\
  &  =  & K_2\epsilon^\lambda\sum_{n=m}^\infty \alpha^{(k\lambda-1)n}
            n^{k\lambda-d}\exp(-n) = +\infty, \nonumber
\end{eqnarray}
where $|T_n|$ is the Lebesgue measure of $T_n$, and this completes the proof.
\end{proof}

As (\ref{eq:ldera}) shows, no amount of ``derivative sacrifice'' will overcome this
property of $G_f$: there is no choice of $2\le m<k$ such that
$\psi(a)\in W^{m,(\lambda,\ldots,\lambda)}(\mu)$ for all $a\in G_f$.  (Change $k$ to
$m$ in the definition of $\alpha$.)  However, we are able to prove the following,
which includes the case $k=2$, $\lambda=\nu=1$.

\begin{prop} \label{pr:fixedn}
Let $k\ge 2$, let $\lambda\ge k-1$ and let $\nu:=(\lambda+1)/k$.
\begin{enumerate}
\item[(i)] For any $a\in G_f$, $\psi(a)\in W^{k,(\nu,\ldots,\nu)}(\mu)$.
\item[(ii)] The nonlinear superposition operator
  $\Psi_f: G_f\rightarrow W^{k,(\nu,\ldots,\nu)}(\mu)$, defined by
  $\Psi_f(a)(x)=\psi(a(x))$, is continuous.
\end{enumerate}
\end{prop}

\begin{proof}
As in the proof of Proposition \ref{pr:rapsi}, it suffices to show that, for any
$a\in G_f$, any sequence $(a_n\in G_f)$ converging to $a$ in $G_f$, and any
$s\in S_0$, $F_s(a_n)\rightarrow F_s(a)$ in $L^\nu(\mu)$, where $F_s$ is as defined
in (\ref{eq:FdiB}).  For any $s$ with $|s|<k$ this can be accomplished by
means of H\"{o}lder's inequality, as in the proof of Proposition \ref{pr:rapsi}.
Furthermore, even if $|s|=k$, all terms in the sum on the right-hand side of
(\ref{eq:FdiB}) for which $|\pi|<k$ can be treated in the same way.  (There are no
more than $k-1$ factors in the product, each of which is in $L^\lambda(\mu)$, and
$\lambda/(k-1)\ge \nu$.)  This leaves the terms for which $|\pi|=|s|=k$; in order
to show that these converge in $L^\nu(\mu)$ it suffices to show that, for any
$1\le i\le d$, the sequence $(|\psi^{(k)}(a_n)(a_n^\prime)^k|^\nu)$
is uniformly integrable, where, for any weakly differentiable $g:\R^d\rightarrow\R$,
$g^\prime:=\partial g/\partial x_i$.

Let $\rho\in C_0^\infty(\R^d;\R)$ be as defined in (\ref{eq:rhodef}), and let
\[
K_\rho:=\textstyle \sup_x(\rho(x)+2|\rho^\prime(x)|+|\rho^{\prime\prime}(x)|).
\]
Let $h:G_f\rightarrow L^1(\mu)$ be defined by
$h(a)=|a|+|a^\prime|+(|a|+|a^\prime|+|a^{\prime\prime}|)^\lambda$; then
$h(K_\rho a_n)\rightarrow h(K_\rho a)$ in $L^1(\mu)$ and so, according to the
Lebesgue-Vitaly theorem, $(h(K_\rho a_n))$ is a uniformly integrable
sequence.  So, according to the de la Vall\'{e}e Poussin theorem, there exists a
convex increasing function $\Ftil:[0,\infty)\rightarrow[0,\infty)$ such that
$\Htil(z):=\Ftil(z)/z$ is an unbounded, non-decreasing function and
$\sup_n\Emu\Ftil(h(K_\rho a_n))<\infty$.  Let $H:[0,\infty)\rightarrow[0,\infty)$
be defined by
\begin{equation}
H(z) = \left\{\begin{array}{ll} 0 & {\rm if\ }z=0 \\
         z^{-1}\int_0^z \Htil(y)\,dy & {\rm otherwise}.\end{array} \right\}
    \le \Htil(z) \label{eq:newHdef}
\end{equation}
For any $y\in[0,\infty)$, let $z_y:=\inf\{z\in[0,\infty):\Htil(z)\ge y\}$; for
any $z>2z_y$,
\[
H(z) = z^{-1}\int_0^{z_y}\Htil(t)\,dt + z^{-1}\int_{z_y}^z\Htil(t)\,dt
       \ge (z-z_y)y/z \ge y/2,
\]
and so $H$ is also unbounded.  Furthermore
\begin{equation}
zH^{(1)}(z) = \Htil(z)-H(z)\in[0,\Htil(z)]. \label{eq:zHonebnd}
\end{equation}
Summarising the above, $H$ is unbounded, non-decreasing and differentiable, and so
$F:[0,\infty)\rightarrow[0,\infty)$, defined by $F(z)=zH(z)$ is another de la
Vall\'{e}e Poussin function for which $\sup_n\Emu F(h(K_\rho a_n))<\infty$.

Let $G:[0,\infty)\rightarrow[0,\infty)$ be defined by $G(z)=zH(|z/C|^{1/k\nu})$,
where $C:=\sup_z|\psi^{(k)}(z)|^\nu$; then, for any $f\in C_0^\infty(\R^d;\R)$,
\begin{eqnarray}
\Emu G\left(|\psi^{(k)}(f)|^\nu |f^\prime|^{k\nu}\right)
  & \le & K_1\Emu|\psi^{(k)}(f)||f^\prime|^{k\nu}H(|f^\prime|) \nonumber \\
  & \le & K_2\Emu\frac{\psi^{(1)}(f)}{(1+\psi(f))^k}|f^\prime|^{k\nu} H(|f^\prime|)
          \nonumber \\
  &  =  & K_2\int\frac{\psi(f)^\prime}{(1+\psi(f))^k}f^\prime
            |f^\prime|^{k\nu-2}H(|f^\prime|)r\, dx \nonumber \\
  &  =  & K_3\int(1+\psi(f))^{1-k}\bigg[
            |f^\prime|^{k\nu-2}f^{\prime\prime}H(|f^\prime|) \label{eq:parts1} \\
  &     & \quad + |f^\prime|^{k\nu-1}H^{(1)}(|f^\prime|)f^{\prime\prime}
            + f^\prime|f^\prime|^{k\nu-2}H(|f^\prime|)l^\prime\bigg]r\, dx
            \nonumber \\
  & \le & K_3(R(f)+S(f)+T(f)), \nonumber
\end{eqnarray}
where $K_1$ and $K_2$ depend only on the function $\psi$, $K_3/K_2=(k\nu-1)/(k-1)$
and $R(f)$, $S(f)$ and $T(f)$ are as follows:
\begin{eqnarray}
R(f)
  & := & \Emu|f^\prime|^{k\nu-2}|f^{\prime\prime}|H(|f^\prime|),\nonumber \\
S(f)
  & := & \Emu|f^\prime|^{k\nu-1}H^{(1)}(|f^\prime|)|f^{\prime\prime}|
         \le \Emu|f^\prime|^{k\nu-2}|f^{\prime\prime}|\Htil(|f^\prime|),
         \label{eq:Sbnd} \\
T(f)
  & := & \Emu|f^\prime|^{k\nu-1}H(|f^\prime|)|l^\prime|. \nonumber
\end{eqnarray}
In (\ref{eq:parts1}), we have used the boundedness of $\psi^{(k)}$ in the first step,
Lemma \ref{le:psidif}(ii) in the second step and integration by parts with respect
to $x_i$ in the fourth step. (If $t=1$ in Example \ref{ex:muex}(i), then
$\theta_t(|\cdot|)$ is not differentiable at $0$ and the integration by parts has to
be accomplished separately on the two sub-intervals $(-\infty,0)$ and $(0,\infty)$.)
In (\ref{eq:Sbnd}), we have used (\ref{eq:zHonebnd}).  Let
$a_{m,n}:=a_n(x)\rho(x/m)\in\cu_m$; then, with $J$ as defined in section
\ref{se:modspa},
\begin{eqnarray*}
R(Ja_{m,n})
  & \le & \Emu F\left(|(Ja_{m,n})^\prime|
            +(|(Ja_{m,n})^\prime|+|(Ja_{m,n})^{\prime\prime}|)^\lambda\right) \\
  &  =  & \Emu F\left(|Ja_{m,n}^\prime|
            +(|Ja_{m,n}^\prime|+|Ja_{m,n}^{\prime\prime}|)^\lambda\right) \\
  & \le & \Emu JF\left(|a_{m,n}^\prime|
            +(|a_{m,n}^\prime|+|a_{m,n}^{\prime\prime}|)^\lambda\right) \\
  & \le & \Emu F\left(|a_{m,n}^\prime|
            +(|a_{m,n}^\prime|+|a_{m,n}^{\prime\prime}|)^\lambda\right) + 1/m \\
  & \le & \Emu F(h(K_\rho a_n)) + 1/m,
\end{eqnarray*}
where we have used the definition of $F$ in the first step, Lemma
\ref{le:molprop}(ii) in the second step, Jensen's inequality in the third step,
Lemma \ref{le:molprop}(i) in the fourth step and (\ref{eq:Dsfrho}) in the final
step.  Similar bounds can be found for $S(Ja_{m,n})$ and, if $t\in(0,1]$ (so that
$l^\prime$ is bounded), $T(Ja_{m,n})$.

If $t\in(1,2]$ we note that
\[
\left(|f^\prime|^\lambda H(|f^\prime|)\right)^\prime
  = \lambda|f^\prime|^{\lambda-1}{\rm sgn}(f^\prime)f^{\prime\prime}H(|f^\prime|)
    + |f^\prime|^\lambda H^{(1)}(|f^\prime|){\rm sgn}(f^\prime)f^{\prime\prime},
\]
so that $\Emu|(|f^\prime|^\lambda H(|f^\prime|))^\prime|\le\lambda R(f) + S(f)$,
and
\[
\||(Ja_{m,n})^\prime|^\lambda H(|(Ja_{m,n})^\prime|)\|_{W^{1,(1,1)}(\mu)}
  \le (\lambda+2)(\Emu\Ftil(h(K_\rho a_n))+1/m).
\]
Let $\alpha:=(t-1)/t$, and let $L^1\log^\alpha L(\mu)$ and $\exp L^{1/\alpha}(\mu)$
be the complementary Orlicz spaces defined in the proof of Lemma \ref{le:superp}.
It follows from Theorem 7.12 in \cite{cipi1} that, for some $K_4<\infty$ not
depending on $m$ or $a_n$,
\[
\||(Ja_{m,n})^\prime|^\lambda H(|(Ja_{m,n})^\prime|)\|_{L^1\log^\alpha L(\mu)}
  \le K_4(\lambda+2)(\Emu\Ftil(h(K_\rho a_n))+1/m).
\]
For any $|x_i|>z_t$, $l^\prime(x)=-t|x_i|^{t-1}{\rm sgn}(x_i)$, and so
$l^\prime\in\exp L^{1/\alpha}(\mu)$, and the generalised H\"{o}lder
inequality shows that, for some $K_5<\infty$
\[
T(Ja_{m,n})
 \le K_5\||(Ja_{m,n})^\prime|^\lambda H(|(Ja_{m,n})^\prime|)\|_{L^1\log^\alpha L(\mu)}
     \|l^\prime\|_{\exp L^{1/\alpha}(\mu)}.
\]
We have thus shown that, for any $t\in(0,2]$,
\begin{equation}
\Emu G\left(|\psi^{(k)}(Ja_{m,n})|^\nu |(Ja_{m,n})^\prime|^{k\nu}\right)
  \le K_6(\Emu\Ftil(h(K_\rho a_n)) + 1/m), \label{eq:Gbnd}
\end{equation}
where $K_6<\infty$ does not depend on $m$ or $a_n$.  Since $G$ is a de la
Vall\'{e}e Poussin function, the sequence
$(|\psi^{(k)}(Ja_{m,n})|^\nu |(Ja_{m,n})^\prime|^{k\nu},\,m\in\N)$, for any fixed
$n$, is uniformly integrable and so, according to Lemma \ref{le:monoleb}, converges
in $L^1(\mu)$ to $|\psi^{(k)}(a_n)|^\nu |a_n^\prime|^{k\nu}$.  Fatou's theorem now
shows that
\[
\Emu G\left(|\psi^{(k)}(a_n)|^\nu |a_n^\prime|^{k\nu}\right)
  \le K_6\Emu\Ftil(h(K_\rho a_n)),
\]
which in turn shows that the sequence
$(|\psi^{(k)}(a_n)|^\nu |a_n^\prime|^{k\nu},\,n\in\N)$ is uniformly integrable.
So $\psi^{(k)}(a_n)(a_n^\prime)^k\rightarrow\psi^{(k)}(a)(a^\prime)^k$ in
$L^\nu(\mu)$, which completes the proof of part (ii).
\end{proof}

If we want all derivatives of $\psi(a)$ to be continuous maps from $G_f$ to
$L^\nu(\mu)$ (for some $\nu\ge 1$) then the fixed norm space $G_f$ should have
Lebesgue exponent $\lambda=\max\{2,\nu k-1\}$. (The resulting manifold will not
have a strong enough topology for global information geometry unless $\lambda\ge 2$.)
The mixed norm space $G_m$  requires $\lambda_1=\nu k$, $\lambda_2=\nu k/2$,
\ldots, $\lambda_k=\nu$.
This places a slightly higher integrability constraint on the first derivative, but
lower constraints on all other derivatives (significantly lower if $k\ge 3$).
Furthermore, if $G_f$ is used as a model space, then $\psi(a)$ and its first partial
derivatives actually belong to $L^\lambda(\mu)$, and so the true range of the
superposition operator in this context is a mixed norm space, whether or not we
choose to think about it in this way.

The case in which $\lambda=1$ is of particular interest.  Proposition \ref{pr:fixedn}
then shows that $\psi$ defines a nonlinear superposition operator that acts
continuously on $G_s:=W^{2,(1,1,1)}(\mu)$.  The use of such a low Lebesgue
exponent precludes the results in section \ref{se:manifmt} concerning the
smoothness of the KL-divergence.  In particular, we cannot expect to retain global
geometric constructs such as the Fisher-Rao metric.  However,
$\cd(\mu|\fndot):M_s\rightarrow[0,\infty)$ is still continuous for all $t\in(0,2]$,
and $\cd(\fndot|\mu)$ is finite if $t=2$.  Since $\psi^{(1)}$ is bounded, there is
no difficulty in extending these results as follows.

\begin{corol}
For any $\lambda_0\in[1,\infty)$, $\psi$ defines a nonlinear superposition operator
that acts continuously on $G_{ms}:=W^{2,(\lambda_0,1,1)}(\mu)$.
\end{corol}

\begin{remark}
When the model space, $G$, is $\Gtil$, $G_s$ or $G_{ms}$, then condition (M2) can
be replaced by: (M2') $p,\log p\in G$.
\end{remark}

\section{The Manifolds of Probability Measures} \label{se:manifm}

In this section we shall assume that $\lambda_0>1$, or that $\lambda_0=1$ and the
embedding hypothesis (E1) holds.  Let $M_0\subset M$ be the subset of the general manifold of section \ref{se:manifmt} (that modelled on $G:=W^{k,\Lambda}(\mu)$),
whose members are {\em probability} measures.  These satisfy the additional
hypothesis:
\begin{enumerate}
\item[(M3)] $\Emu p = 1$.
\end{enumerate}
The co-dimension 1 subspaces of $L^\lambda(\mu)$ and $G$, whose members,
$a$, satisfy $\Emu a=0$ will be denoted $L_0^\lambda(\mu)$, and $G_0$.
Let $\phi_0:M_0\rightarrow G_0$ be defined by
\begin{equation}
\phi_0(P) = \phi(P) - \Emu\phi(P)
   = \log_d p - \Emu\log_d p.  \label{eq:phidef}
\end{equation}

\begin{prop} \label{pr:bijec}
\begin{enumerate}
\item[(i)] $\phi_0$ is a bijection onto $G_0$.  Its inverse takes the form
  \begin{equation}
  p(x) = \frac{d\phi_0^{-1}(a)}{d\mu}(x) = \psi(a(x)+Z(a)), \label{eq:phiinv}
  \end{equation}
  where $Z\in C^N(G_0;\R)$ is an (implicitly defined) normalisation function,
  and $N=N(\lambda_0,\lambda_1,1,t)$ is as defined in (\ref{eq:Ndef}).
\item[(ii)] The first (and if $N\ge 2$ second) derivative of $Z$ is as follows:
  \begin{eqnarray}
  Z_a^{(1)}u & = & -\E_{P_a}u  \nonumber \\
          [-1.5ex] \label{eq:Zder} \\ [-1.5ex]
  Z_a^{(2)}(u,v)
    & = & -\frac{\Emu\psi^{(2)}(a+Z(a))(u-\E_{P_a}u)(v-\E_{P_a}v)}
           {\Emu\psi^{(1)}(a+Z(a))}, \nonumber
  \end{eqnarray}
  where $P_a:=\Ptil_a/\Ptil_a(\R^d)$ and $\Ptil_a$ is the finite measure defined
  before (\ref{eq:bunchartt}).
\item[(iii)] If $\lambda_0-1\in\N$ and (E1) does not hold then $Z^{(\lambda_0-1)}$
  is Leslie differentiable (with derivative is as in (\ref{eq:Zder}) if
  $\lambda_0=2$).
\item[(iv)] $Z$ and any derivatives it admits are bounded on bounded sets.
\end{enumerate}
\end{prop}

\begin{proof}
Let $\Upsilon:G_0\times\R\rightarrow(0,\infty)$ be defined by
\begin{equation}
\Upsilon(a,z) = \Emu\psi(a+z) = \Emu\Psi_1(a+z),
\end{equation}
where $\Psi_\beta$ is as in Lemma \ref{le:superp}.  It follows from Lemma
\ref{le:superp}, that $\Upsilon$ is of class $C^N$ and that, for any $u\in G_0$,
\begin{equation}
\Upsilon_{a,z}^{(1,0)}u = \Emu\psi^{(1)}(a+z)u \quad{\rm and}\quad
\Upsilon_{a,z}^{(0,1)} = \Emu\psi^{(1)}(a+z) > 0. \label{eq:Upsder}
\end{equation}
Since $\psi$ is convex,
\[
\sup_z\Upsilon(a,z) \ge \sup_z\psi(\Emu(a+z)) = \sup_z\psi(z) = +\infty;
\]
furthermore, the monotone convergence theorem shows that
\[
\lim_{z\downarrow-\infty} \Upsilon(a,z) = \Emu\lim_{z\downarrow-\infty}\psi(a+z) = 0.
\]
So $\Upsilon(a,\fndot)$ is a bijection with strictly positive derivative, and the
inverse function theorem shows that it is a $C^N$-isomorphism.  The implicit
mapping theorem shows that $Z:G_0\rightarrow\R$, defined by
$Z(a)=\Upsilon(a,\fndot)^{-1}(1)$, is of class $C^N$.  For some $a\in G_0$, let $P$
be the probability measure on $\cx$ with density $p=\psi(a+Z(a))$; then $\phi_0(P)=a$
and $P\in M_0$, which completes the proof of part (i).

That the first derivative of $Z$ is as in (\ref{eq:Zder}) follows from
(\ref{eq:Upsder}).  Since $\Emu\psi^{(1)}(a+Z(a))>0$, parts (ii) and (iii) follow
from Lemma \ref{le:superp} and the chain and quotient rules of differentiation
(which hold for Leslie derivatives).  Part (iv) is proved in Proposition 4.1 in
\cite{newt6}. 
\end{proof}

Expressed in charts, the inclusion map $\imath:M_0\rightarrow M$ is as follows
\begin{equation}
\rho(a) := \phi\circ\phi_0^{-1}(a) = a + Z(a),  \label{eq:inclus}
\end{equation}
and has the same smoothness properties as $Z$.  The following goes further.

\begin{prop} \label{pr:embed}
$(M_0,G_0,\phi_0)$ is a $C^N$-embedded submanifold of $(M,G,\phi)$, where
  $N=N(\lambda_0,\lambda_1,1,t)$ is as defined in (\ref{eq:Ndef}).
\end{prop}

\begin{proof}
Let $\eta:G\rightarrow G_0$ be the superposition operator defined by
$\eta(a)(x)=a(x)-\Emu a$; then $\eta$ is of class $C^\infty$, has first
derivative $\eta_a^{(1)}u=u-\Emu u$, and zero higher derivatives.
Now $\eta\circ\rho$ is the identity map of $G_0$, which shows that $\rho$ is
homeomorphic onto its image, $\rho(G_0)$, endowed with the relative topology.
Furthermore, for any $u\in G_0$,
\[
u=(\eta\circ\rho)_a^{(1)}u=\eta_{\rho(a)}^{(1)}\rho_a^{(1)}u,
\]
and so $\rho_a^{(1)}$ is a toplinear isomorphism, and its image, $\rho_a^{(1)}G_0$,
is a closed linear subspace of $G$.  Let $E_a$ be the one dimensional
subspace of $G$ defined by $E_a=\{y\psi^{(1)}(\rho(a)): y\in\R\}$.  If $u\in E_a$
and $v\in\rho_a^{(1)}G_0$ then there exist $y\in\R$ and $w\in G_0$ such that
\[
\Emu uv = y\Emu\psi^{(1)}(\rho(a))(w-\E_{P_a}w) = 0.
\]
So $E_a\cap\rho_a^{(1)}G_0=\{0\}$, and $\rho_a^{(1)}$ {\em splits} $G$ into
the direct sum $E_a\oplus \rho_a^{(1)}G_0$.  We have thus shown that $\rho$ is a
$C^N$-immersion, and this completes the proof.
\end{proof}

For any $P\in M_0$, the tangent space $T_PM_0$ is a subspace of $T_PM$ of
co-dimension 1; in fact, as shown in the proof of Proposition \ref{pr:embed},
\begin{equation}
T_PM = T_PM_0 \oplus \{y\hat{U},\,y\in\R\}, \quad
       {\rm where\ }\hat{U}\phi=\psi^{(1)}(\phi(P)). \label{eq:tansplit}
\end{equation}
Let $\Phi_0:TM_0\rightarrow G_0\times G_0 $ be defined as follows:
\begin{equation}
\Phi_0(P,U) = \Phi(P,U)-\Emu\Phi(P,U). \label{eq:bunchart}
\end{equation}
Then $\Phi\circ\Phi_0^{-1}(a,u)=(\rho(a),\rho_a^{(1)}u)$.  For any $(P,U)\in TM_0$,
$U\phi=\rho_a^{(1)}u=u-\E_{P_a}u$, and so tangent vectors in $T_PM_0$ are
distinguished from those merely in $T_PM$ by the fact that their total mass is zero:
\begin{equation}
U(\R^d) = \int (u-\E_{P_a}u)d\Ptil_a = 0.  \label{eq:zerowe}
\end{equation}

The map $Z$ of (\ref{eq:phiinv}) is (the negative of) the additive normalisation
function, $\alpha$, associated with the interpretation of $M_0$ as a {\em generalised
exponential model} with deformed exponential function $\psi$.  (See Chapter 10 in
\cite{naud1}.  We use the symbol $Z$ rather than $-\alpha$ for reasons of consistency
with \cite{newt4,newt6}.)  In this context, the probability measure
$P_a$ of (\ref{eq:Zder}) is called the {\em escort} measure to $P$.  In \cite{mopi1},
the authors considered {\em local} charts on the Hilbert manifold of \cite{newt4}. 
In the present context, these take the form $\phi_P:M_0\rightarrow G_P$, where $G_P$
is the subspace of $G$ whose members, $b$, satisfy $\E_{P_a}b=0$.  This amounts to
re-defining the origin of $G$ as $\phi(P)$, and using the co-dimension 1 subspace
that is tangential to the image $\phi(M_0)$ at this new origin as the model space.
This local chart is {\em normal} at $P$ for the Riemannian metric and Levi-Civita
parallel transport induced by the global chart $\phi$ on $M$.  However, the metric
differs from the Fisher-Rao metric on all fibres of the tangent bundle other than
that at $\mu$.

The equivalent on $M_0$ of the maps $m_\beta$ and $e_\beta$ of section
\ref{se:manifmt} are the maps
$m_{\beta,0}, e_{\beta,0}:M_0\rightarrow L_0^\beta(\mu)$, defined by
\begin{equation}
m_{\beta,0}(P) = m_\beta(P)\quad{\rm and}\quad
e_{\beta,0}(P) = e_\beta(P)-\Emu e_\beta(P).  \label{eq:medef}
\end{equation}
Their properties are developed in \cite{newt4}, and follow from those of $m_\beta$
and $e_\beta$.

\section{Application to Nonlinear Filtering} \label{se:NLF}

We sketch here an application of the manifolds of sections \ref{se:manifmt} and
\ref{se:manifm} to the nonlinear filtering problem discussed in section \ref{se:intro}.
An abstract filtering problem (in which $X$ is a Markov process evolving
on a measurable space) was investigated in \cite{newt5}.  Under suitable technical
conditions, it was shown that the $(Y_s,\,0\le s\le t)$-conditional distribution
of $X_t$, $\Pi_t$, satisfies an infinite-dimensional stochastic differential equation
on the Hilbert manifold of \cite{newt4}, and this representation was used to study
the filter's information-theoretic properties.  This equation involves the
normalisation constant $Z$, which is difficult to use since it is implicitly defined,
and so it is of interest to use a manifold of finite measures not involving $Z$,
such as $M$ of section \ref{se:manifmt}.  Because of its special connection with
the function $\psi$, the mixed norm model space $\Gtil$ of section \ref{se:gtil} is
of particular interest, although the fixed norm spaces of section \ref{se:fixnor}
could also be used.

If the conditional distribution $\Pi_t$ has a density with respect to Lebesgue
measure, $p_t$, satisfying the Kushner-Stratonovich equation (\ref{eq:kusteq}), then
its density with respect to $\mu$, $\pi_t=p_t/r$, also satisfies (\ref{eq:kusteq}),
but with the transformed forward operator:
\begin{equation}
\ca\pi = \frac{1}{2r}\frac{\partial^2\Gamma^{ij}r\pi}{\partial x^i\partial x^j}
          - \frac{1}{r}\frac{\partial f^ir\pi}{\partial x^i},  \label{eq:frwdop}
\end{equation}
where $\Gamma=gg^*$ and we have used the Einstein summation convention. The density
$\pi_t$ also satisfies
\begin{equation}
d\pi_t
  = \ca \pi_t\,dt + \pi_t (h-\hhbar(\pi_t))(dY_t-\hhbar(\pi_t)dt),  \label{eq:nornlf}
\end{equation}
where, for appropriate densities $p$, $\hhbar(p):=(\Emu p)^{-1}\Emu ph$.
Unlike (\ref{eq:kusteq}), this equation is {\em homogeneous}, in the sense that
if $\pi_t$ is a solution then so is $\alpha \pi_t$, for any $\alpha>0$.  A
straightforward formal calculation shows that $\log_d\pi_t$ satisfies the following
stochastic partial differential equation
\begin{equation}
da_t = \bfu(\fndot,a_t)dt
       + \bfv(\fndot,a_t)(dY_t-\hhbar(\psi(a_t))dt),\label{eq:gtilnlf}
\end{equation}
where
\begin{eqnarray}
\bfv(x,a)
  & = & (1+\psi(a(x)))(h(x)-\hhbar(\psi(a))), \nonumber \\
\bfu(x,a)
  & = & \half\Gamma^{ij}(x)\left[\frac{\partial^2a}{\partial x^i\partial x^j}(x)
        + (1+\psi(a(x)))^{-2}\frac{\partial a}{\partial x^i}(x)
          \frac{\partial a}{\partial x^j}(x)\right] \label{eq:frwdphi} \\
  &   & \quad + F^i(x)\frac{\partial a}{\partial x^i}(x) + (1+\psi(a(x)))F^0(x)
        -\half[h(x)-\hhbar(\psi(a))]^2,  \nonumber
\end{eqnarray}
and
\begin{eqnarray*}
F^i
  & = & \Gamma^{ij}\frac{\partial l}{\partial x^j}
           + \frac{\partial \Gamma^{ij}}{\partial x^j} - f^i, \nonumber \\
F^0
  & = & \half\frac{\partial^2\Gamma^{ij}}{\partial x^i\partial x^j}
           + \frac{\partial\Gamma^{ij}}{\partial x^i}\frac{\partial l}{\partial x^j}
           + \half\Gamma^{ij}\left[\frac{\partial^2l}{\partial x^i\partial x^j}
             + \frac{\partial l}{\partial x^i}\frac{\partial l}{\partial x^j}\right]
           - f^i\frac{\partial l}{\partial x^i} - \frac{\partial f^i}{\partial x^i}.
           \nonumber
\end{eqnarray*}
In order to make sense of (\ref{eq:gtilnlf}) and (\ref{eq:frwdphi}), we need further
hypotheses.  The following are used for illustration purposes, and are not intended
to be ripe.
\begin{enumerate}
\item[(F1)] $\mu$ is the smooth distribution of Example \ref{ex:muex}(ii) with $t=1$,
  and $\lambda_0\ge 2$.
\item[(F2)] The functions $f$, $g$ and $h$ are of class $C^\infty(\R^d)$.
\item[(F3)] The functions $f$ and $h$, and all their derivatives, satisfy polynomial
  growth conditions in $|x|$.
\item[(F4)] The function $g$ and all its derivatives are bounded.
\end{enumerate}
In particular, these allow $\hhbar$, $\bfu$ and $\bfv$ to be defined on $M$ in
a precise way.

\begin{prop} \label{pr:coefrep}
\begin{enumerate}
\item[(i)] The functional $\Hbar:\Gtil\rightarrow\R$, defined by
  $\Hbar(a)=\hhbar(\psi(a))$, is of class $C^{\lceil\lambda_0\rceil-1}$.
\item[(ii)] Let $k\ge 2$ and $\lambda_1\ge 2k$.  If $a\in\Gtil$ then
  $\bfu(\fndot,a),\bfv(\fndot,a)\in H^{k-2}(\mu)$, where $H^{k-2}(\mu)$ is the Hilbert
  Sobolev space of Remark \ref{re:hilbert}.
\item[(iii)] The superposition operators $\bfU,\bfV:\Gtil\rightarrow H^{k-2}(\mu)$,
  defined by $\bfU(a)(x)=\bfu(x,a)$ and $\bfV(a)(x)=\bfv(x,a)$, are continuous.
\end{enumerate}
\end{prop}

\begin{proof}
It follows from (F1--F4) that
\begin{equation}
F^i,F^0,h,\in W^{k,(\lambda,\lambda,\ldots,\lambda)}(\mu)\quad
  {\rm for\ every\ }k\in\N,{\rm\ and\ every\ }\lambda\in[1,\infty). \label{eq:FGhbnd}
\end{equation}
Lemma \ref{le:superp} shows that, for any $\epsilon>0$, $\Psi_{1+\epsilon}$ is of
class $C^{\lceil\lambda_0/(1+\epsilon)\rceil-1}$.  For any $\lambda_0\in[2,\infty)$
there exists an $\epsilon>0$ such that
$\lceil\lambda_0/(1+\epsilon)\rceil=\lceil\lambda_0\rceil$ and so with this choice,
$\Psi_{1+\epsilon}$ is of class $C^{\lceil\lambda_0\rceil-1}$.
H\"{o}lder's inequality shows that, for any $0\le i\le\lceil\lambda_0\rceil-2$, any
$a,b\in\Gtil$ and any $u_1,\ldots, u_i$ in the unit ball of $\Gtil$,
\begin{eqnarray*}
\Emu\big|(\Psi_{1+\epsilon,b}^{(i)}-\Psi_{1+\epsilon,a}^{(i)}
  -\Psi_{1+\epsilon,a}^{(i+1)}(b-a))(u_1,\ldots,u_i)h\big| \qquad\qquad\qquad\qquad\\
  \qquad
  \le \big\|\Psi_{1+\epsilon,b}^{(i)}-\Psi_{1+\epsilon,a}^{(i)}
  -\Psi_{1+\epsilon,a}^{(i+1)}(b-a)\big\|_{L(\Gtil^i;L^{1+\epsilon}(\mu))}
    \|h\|_{L^{(1+\epsilon)/\epsilon}(\mu)}, \\
\Emu\big|(\Psi_{1+\epsilon,b}^{(i+1)}-\Psi_{1+\epsilon,a}^{(i+1)})h\big|
  \le \big\|\Psi_{1+\epsilon,b}^{(i+1)}-\Psi_{1+\epsilon,a}^{(i+1)}
        \big\|_{L(\Gtil^{i+1};L^{1+\epsilon}(\mu))}
        \|h\|_{L^{(1+\epsilon)/\epsilon}(\mu)},
\end{eqnarray*}
which shows that the map $\Gtil\ni a\mapsto\Emu\psi(a)h\in\R$ is of class
$C^{\lceil\lambda_0\rceil-1}$.  The quotient rule of differentiation and the fact
that $\Emu\psi(a)>0$ complete the proof of part (i).

Parts (ii) and (iii) can be proved by applying H\"{o}lder's inequality to the weak
derivatives of the various components of $\bfu(\fndot,a)$ and $\bfv(\fndot,a)$.  The
quadratic term in $\bfu$ is the most difficult to treat, and so we give a detailed
proof for this.  We begin by noting that
$(1+\psi(a))^{-1}\partial a/\partial x^i=\partial(a-\psi(a))/\partial x^i$.  For any
$|s|\le k-2$
\begin{equation}
D^s\frac{\partial\psi(a)}{\partial x^i}\frac{\partial\psi(a)}{\partial x^j}
  = \sum_{\sigma\le s}D^\sigma\frac{\partial\psi(a)}{\partial x^i}
    D^{s-\sigma}\frac{\partial\psi(a)}{\partial x^j}
    \prod_{1\le l\le d}\left(\begin{array}{c} s_l \\ \sigma_l\end{array}\right).
    \label{eq:proddif}
\end{equation}
According to Proposition \ref{pr:rapsi}, the nonlinear superposition operator
$\Psi_{\sigma,i}:\Gtil\rightarrow L^{\lambda_1/(|\sigma|+1)}(\mu)$ defined by
$\Psi_{\sigma,i}(a)=D^\sigma(\partial\psi(a)/\partial x^i)$ is continuous, and so it
follows from H\"{o}lder's inequality that the same is true of
$\Upsilon_{s,i,j}:\Gtil\rightarrow L^{\lambda_1/(|s|+2)}(\mu)$ defined by
the right-hand side of (\ref{eq:proddif}).  Together with (F4), this shows that
$\Upsilon:\Gtil\rightarrow H^{k-2}(\mu)$ defined by
$\Upsilon(a)=\Gamma^{ij}(\partial\psi(a)/\partial x^i)(\partial\psi(a)/\partial x^j)$
is continuous.

The other components of $\bfu(\fndot,a)$ and the only component of $\bfv(\fndot,a)$
can be shown to have the stated continuity by similar arguments.  These make use of
(\ref{eq:FGhbnd}), Proposition \ref{pr:rapsi} and part (i) here.
\end{proof}

\begin{remark}
There are many variants of this proposition, corresponding to different choices of
the domain and range of $\bfU$ and $\bfV$.  If $\lambda_0$ and $\lambda_1$ are
suitably large, then $\bfU$ and $\bfV$ admit various derivatives on $M$.
\end{remark}

One application of Proposition \ref{pr:coefrep} is in the development of
{\em projective approximations}, as proposed in the context of the exponential Orlicz
manifold in \cite{brpi1} and the earlier references therein.  As a particular
instance, suppose that $k\ge 2$ and $\lambda_1\ge 2k$; let
$(\eta_i\in C^k(\R^d)\cap\Gtil,\,1\le i\le m)$ be linearly independent, and define
\begin{equation}
G_{m,\eta} = \{a\in\Gtil:\,a=\alpha^i\eta_i{\rm\ for\ some\ }\alpha\in\R^m\}. \label{eq:fdspa}
\end{equation}
This is an $m$-dimensional linear subspace of both $\Gtil$ and $H^{k-2}(\mu)$.
We can use the inner product of $H^{k-2}(\mu)$ to project members of
$H^{k-2}(\mu)$ onto $G_{m,\eta}$.  In particular, we can project $\bfU(a)$ and
$\bfV(a)$ onto $G_{M,\eta}$ for any $a\in G_{m,\eta}$ to obtain continuous vector
fields of the finite-dimensional submanifold of $M$ defined by
$M_\eta=\phi^{-1}(G_{m,\eta})$.  Since the model space norms of $H^{k-2}(\mu)$
dominate the Fisher-Rao metric on every fibre of the tangent bundle
(\ref{eq:fishdom}), the projection takes account of the information theoretic cost
of approximation, as well as controlling the derivatives of the conditional density
$\pi_t$.

$M_\eta$ is a finite-dimensional {\em deformed exponential model}, and is trivially a
$C^\infty$-embedded submanifold of $M$.  Many other classes of finite-dimensional
manifold also have this property.  For example, since $\Psi(\Gtil)$ is
convex, certain finite-dimensional mixture manifolds modelled on the space
$G_{m,\eta}$, where $\eta_i\in\Psi(G_m)$, are also $C^\infty$-embedded submanifolds
of $M$.  This is also true of particular finite-dimensional exponential models.

\section{Concluding Remarks} \label{se:conclu}

This paper has developed a class of infinite-dimensional statistical manifolds that
use the balanced chart of \cite{newt4,newt6} in conjunction with a variety of
probability spaces of Sobolev type.  It has shown that the mixed-norm space of section
\ref{se:gtil} is especially suited to the balanced chart (and any other chart with
similar properties), in the sense that densities then also belong to this space and
vary continuously on the manifolds.  It has shown that this property is also true of
a particular fixed norm space involving two derivatives, but can be retained for
fixed norm spaces with more than two derivatives only with the loss of Lebesgue
exponent.  The paper has outlined an application of the manifolds to nonlinear
filtering (and hence to the Fokker-Planck equation).  Although motivated by problems
of this type, the manifolds are clearly applicable in other domains, the Boltzmann
equation of statistical mechanics being an obvious candidate.

The deformed exponential function used in the construction of $M$ has {\em linear
growth}, a feature that has recently been shown to be advantageous in {\em quantum
information geometry} \cite{naud2}.  The linear growth arises from the deformed
logarithm of (\ref{eq:phitdef}), which is dominated by the density, $p$, when the
latter is large.  As recently pointed out in \cite{mopi2}, this property is shared
by other deformed exponentials, notably the Kaniadakis $1$-exponential
$\psi_K(z)=z+\sqrt{1+z^2}$.  The corresponding deformed logarithm is
$\log_K(y)=(y^2-1)/2y$, and so the density is controlled (when close to zero) by the
term $-1/p$ rather than $\log p$, as used here.  In the non-parametric setting, the
need for both $p$ and $1/p$ to be in $L^{\lambda_0}(\mu)$ places significant
restrictions on membership of the manifold.  If, for example, the reference measure
of Example \ref{ex:muex}(i) is used, and $t=1$, then the measure having density
$C\exp(-\alpha|x|)$ (with respect to Lebesgue measure) belongs to the manifold only
if $|\alpha-1|<1/\lambda_0$.

The Kaniadakis $1$-exponential shares the properties of $\psi$ used in this paper;
these are summarised in Lemma \ref{le:kanexp}, which is easily proved by induction.

\begin{lem} \label{le:kanexp}
\begin{enumerate}
\item[(i)] The Kaniadakis $1$-exponential $\psi_K:\R\rightarrow(0,\infty)$ is
  diffeomorphic; in particular
  \begin{equation}
  \psi_K^{(1)} = \frac{2\psi_K^2}{1+\psi_K^2} > 0\quad {\rm and} \quad
  \psi_K^{(2)} = \frac{8\psi_K^3}{(1+\psi_K^2)^3} > 0, \label{eq:kander}
  \end{equation}
  and so $\psi_K$ is strictly increasing and convex.
\item[(ii)] For any $n\ge 2$,
  \begin{equation}
  \psi_K^{(n)} = \frac{Q_{3(n-2)}(\psi_K)}{(1+\psi_K^2)^{2(n-1)}}\psi_K\psi_K^{(1)},
                   \label{eq:kanderbnd}
  \end{equation}
  where $Q_{3(n-2)}$ is a polynomial of degree no more than $3(n-2)$.  In particular,
  $\psi_K^{(n)}$, $\psi_K^{(n)}/\psi_K$ and $\psi_K^{(n)}/(\psi_K\psi_K^{(1)})$ are
  all bounded.
\end{enumerate}
\end{lem}

We can therefore construct a manifold of finite measures $M_K$, as in
section \ref{se:manifmt}, substituting the chart of (\ref{eq:phitdef}) by
$\phi_K:M_K\rightarrow G$, defined by $\phi_K(P) = \log_Kp$.
The only properties of $\psi$ used in section \ref{se:manifmt} are its strict
positivity, and the boundedness of its derivatives, properties shared by $\psi_K$.
The results in section \ref{se:specms} carry over to $M_K$ with the exception of
Proposition \ref{pr:rapsi}(iii).  Most of these depend only on the boundedness
of the derivatives of $\psi$; however, the integration by parts in (\ref{eq:parts1})
uses (\ref{eq:psidrec}), which can be substituted by (\ref{eq:kanderbnd})
in the case of $M_K$.  The results of section \ref{se:manifm} all carry over to
$M_K$.  Proposition \ref{pr:bijec}(v) depends on the strict positivity of
$\inf_{a\in B}\Emu\psi^{(1)}(a)$ for bounded sets $B\subset G$.  This is also
true of $\psi_K$, since
\begin{equation}
\Emu\psi_K^{(1)}(a)
  \ge 2\exp\left(2\Emu\log p - \Emu\log(1+p^2)\right)
  \ge (\Emu p^{-1})^{-2}/2, \label{eq:emupsik}
\end{equation}
where we have used Jensen's inequality in both steps.

$M_K$ is a
subset of $M$.  Let $\tau:\R\rightarrow\R$ be the ``transition function''
$\tau(z)=\log_d\psi_K(z)$.  All derivatives of $\tau$ are bounded, which explains
why the regularity of the KL-divergence on $M$ carries over to $M_K$.  Furthermore,
it follows from arguments similar to those used in the proof of Proposition
\ref{pr:rapsi} that the superposition operator $T_m:G_m\rightarrow G_m$ defined by
$T_m(a)(x)=\tau(a(x))$ is continuous for any of the mixed norm model spaces of
section \ref{se:gtil}.

The deformed logarithm of (\ref{eq:phitdef}) was chosen in \cite{newt4} because
the resulting manifold is highly inclusive, and suited to the Shannon-Fisher-Rao
information geometry.  In this context, it yields the global bound
(\ref{eq:divsumt}). The Kaniadakis $1$-logarithm is less suited to this geometry,
but more to that generated by the Kaniadakis $1$-entropy, which is of interest in
statistical mechanics.  The full development of this geometry is beyond the scope
of this article.

Condition (\ref{eq:thetadef}) (on the reference measure $\mu$) has to be considered
in the context of (M2), which places upper and lower bounds on the rate at which the
densities of measures in $M$ can decrease as $|x|$ becomes large.  For example, if
all nonsingular Gaussian measures are to belong to $M$, then (M2) requires $r$ to
decay more slowly than a Gaussian density, but more rapidly than a Cauchy density.
Variants of the reference measure $\mu$ with $t\in[1,2)$ may be good choices for
such applications.

\end{document}